\documentclass[a4paper,fleqn,reqno,12pt]{amsart}

\usepackage{color}
\usepackage[utf8]{inputenc}
\usepackage[in]{fullpage}

\usepackage{amsmath,amssymb,amsthm}
\usepackage[alphabetic,nobysame]{amsrefs}
\usepackage[foot]{amsaddr}

\catcode`\@=11

\def\@settitle{%
  \baselineskip14\p@\relax
    {\Large\bfseries%\fontsize{24}{18}\selectfont%
  \@title}}

\def\@setauthors{%
  \begingroup
  \def\thanks{\protect\thanks@warning}%
  \trivlist
  \footnotesize \@topsep45\p@\relax
  \advance\@topsep by -\baselineskip
  \item\relax
  \author@andify\authors
  \def\\{\protect\linebreak}%
  {\sc\fontsize{12}{10}\selectfont\authors}%
  \ifx\@empty\contribs
  \else
    ,\penalty-3 \space \@setcontribs
    \@closetoccontribs
  \fi
  \endtrivlist
  \endgroup
}

\def\@secnumfont{\bfseries}%

\def\section{\@startsection{section}{1}%
  \z@{.7\linespacing\@plus\linespacing}{.5\linespacing}%
  {\normalfont\bf}}

\renewcommand{\BibLabel}{%
    \Hy@raisedlink{\hyper@anchorstart{cite.\CurrentBib}\hyper@anchorend}%
    [\thebib]\hfill%
}

\DefineSimpleKey{bib}{arxiv}

\BibSpec{article}{%
    +{}  {\PrintAuthors}                {author}
    +{,} { \textit}                     {title}
    +{.} { }                            {part}
    +{:} { \textit}                     {subtitle}
    +{,} { \PrintContributions}         {contribution}
    +{.} { \PrintPartials}              {partial}
    +{,} { }                            {journal}
    +{}  { \textbf}                     {volume}
    +{}  { \PrintDatePV}                {date}
    +{,} { \issuetext}                  {number}
    +{,} { \eprintpages}                {pages}
    +{,} { }                            {status}
    +{,} { \PrintDOI}                   {doi}
    +{,} { \arxiv}                      {arxiv}
    +{,} { \w@bsite}                    {url}
    +{}  { \parenthesize}               {language}
    +{}  { \PrintTranslation}           {translation}
    +{;} { \PrintReprint}               {reprint}
    +{.} { }                            {note}
    +{.} {}                             {transition}
    +{}  {\SentenceSpace \PrintReviews} {review}
}

\newcommand{\w@bsite}[1]{\href{http:#1}{{\tt #1}}}
\newcommand{\arxiv}[1]{\href{http://arxiv.org/abs/#1}{{\tt arxiv:\hspace{0pt}#1}}}

\catcode`\@=12

\usepackage{chngcntr} 
\counterwithin{section}{part}

\makeatletter
\def\part{\@startsection{part}{0}%
  \z@{3\linespacing\@plus\linespacing}{2\linespacing}%
  {\normalfont\Large\bfseries}}
\def\l@section{\@tocline{1}{0pt}{0pc}{}{}}
\makeatother

\usepackage{parskip}

\numberwithin{equation}{section} \swapnumbers

\usepackage{tikz}
\usetikzlibrary{calc,arrows}
\usepackage[cmtip,matrix,arrow]{xy}
\SelectTips{cm}{10}

\usepackage{enumitem}
\setlist[enumerate,1]{label=\textit{\alph*)},ref=\textit{\alph*)}}
\setlist[enumerate,2]{label=\textit{\roman*)},ref=\textit{\roman*})}

\usepackage{aliascnt}

\newtheorem{theorem}{Theorem}[section]

\newaliascnt{lemma}{theorem}
\newtheorem{lemma}[lemma]{Lemma}
\aliascntresetthe{lemma}

\newaliascnt{corollary}{theorem}
\newtheorem{corollary}[corollary]{Corollary}
\aliascntresetthe{corollary}

\newaliascnt{proposition}{theorem}
\newtheorem{proposition}[proposition]{Proposition}
\aliascntresetthe{proposition}

\theoremstyle{definition}
\newaliascnt{definition}{theorem}
\newtheorem{definition}[definition]{Definition}
\aliascntresetthe{definition}

\newaliascnt{remarks}{theorem}
\newtheorem{Remarks}[remarks]{Remarks}
\aliascntresetthe{remarks}

\newaliascnt{remark}{theorem}
\newtheorem{remark}[remark]{Remark}
\aliascntresetthe{remark}

\newtheorem*{ack}{Acknowledgment}

\renewcommand{\tilde}{\widetilde}

\newcommand{\bxymatrix}[1]{\vcenter{\xymatrix@=10pt{#1}}}
\newcommand{\cxymatrix}[1]{\vcenter{\xymatrix@=15pt{#1}}}
\newcommand{\dxymatrix}[1]{\vcenter{\xymatrix@=20pt{#1}}}

\def\Ddot(#1,#2;#3){\filldraw (#1,#2) node[above]{\hbox to
    0pt{\hss$#3$\hss}} circle (3pt);}

\def\Dline(#1,#2){\draw[thick] (#1,#2)--+(1,0);}

\def\Drarrow(#1,#2){%
  \draw[thick] (#1,#2+1pt)--+(0.85,0);
  \draw[thick] (#1,#2-1pt)--+(0.85,0);
  \draw[thick] (#1+0.9,#2)--+(-8pt,4pt);
  \draw[thick] (#1+0.9,#2)--+(-8pt,-4pt); }

\def\Dlarrow(#1,#2){%
  \draw[thick] (#1,#2+1pt)--+(-0.85,0);
  \draw[thick] (#1,#2-1pt)--+(-0.85,0);
  \draw[thick] (#1-0.9,#2)--+(8pt,4pt);
  \draw[thick] (#1-0.9,#2)--+(8pt,-4pt); }

\def\Dlrarrow(#1,#2){%
  \draw[thick] (#1+0.15,#2+1pt)--+(0.7,0);
  \draw[thick] (#1+0.15,#2-1pt)--+(0.7,0);
  \draw[thick] (#1+0.9,#2)--+(-8pt,4pt);
  \draw[thick] (#1+0.9,#2)--+(-8pt,-4pt);
  \draw[thick] (#1+0.1,#2)--+(8pt,4pt);
  \draw[thick] (#1+0.1,#2)--+(8pt,-4pt); }

\def\Ddots(#1,#2){%
  \draw[thick] (#1,#2)--+(0.6,0);
  \filldraw (#1+0.8,#2) circle (0.5pt);
  \filldraw (#1+1,#2) circle (0.5pt);
  \filldraw (#1+1.2,#2) circle (0.5pt);
  \draw[thick] (#1+1.4,#2)--(#1+2,#2); }

\newcommand{\sX}{\mathsf{X}}
\newcommand{\sM}{\mathsf{M}}
\newcommand{\cA}{\mathcal{A}}
\newcommand{\cC}{\mathcal{C}}
\newcommand{\cF}{\mathcal{F}}
\newcommand{\cL}{\mathcal{L}}

\newcommand{\cP}{\mathcal{P}}
\newcommand{\cQ}{\mathcal{Q}}
\newcommand{\fa}{\mathfrak{a}}

\newcommand{\fL}{\mathfrak{L}}
\newcommand{\fI}{\mathfrak{I}}
\newcommand{\fk}{\mathfrak{k}}
\newcommand{\fl}{\mathfrak{l}}

\newcommand{\fK}{\mathfrak{K}}
\newcommand{\ft}{\mathfrak{t}}
\newcommand{\fC}{\mathfrak{C}}
\newcommand{\fT}{\mathfrak{T}}

\newcommand{\fS}{\mathfrak{S}}
\newcommand{\CC}{\mathbb{C}}
\newcommand{\HH}{\mathbb{H}}
\newcommand{\RR}{\mathbb{R}}
\newcommand{\ZZ}{\mathbb{Z}}
\newcommand{\sA}{\mathsf{A}}
\newcommand{\sB}{\mathsf{B}}
\newcommand{\sC}{\mathsf{C}}
\newcommand{\sD}{\mathsf{D}}
\newcommand{\sE}{\mathsf{E}}
\newcommand{\sF}{\mathsf{F}}
\newcommand{\sG}{\mathsf{G}}
\renewcommand{\P}{\mathbf{P}}
\newcommand{\Vfa}{{\overline\fa}}
\newcommand{\fq}{{\overline f}}
\newcommand{\Kq}{{\overline K}}
\newcommand{\Mq}{{\overline M}}
\newcommand{\Valpha}{{\overline\alpha}}
\newcommand{\Vbeta}{{\overline\beta}}
\newcommand{\Vpi}{{\overline\pi}}
\newcommand{\Vw}{{\overline w}}
\newcommand{\VW}{{\overline W}}
\newcommand{\VPhi}{{\overline\Phi}}

\newcommand{\thetaq}{{\overline\theta}}

\newcommand{\uq}{{\overline u}}
\renewcommand{\[}{\begin{equation}}
\renewcommand{\]}{\end{equation}}
\newcommand{\<}{\langle}
\renewcommand{\>}{\rangle}
\renewcommand{\epsilon}{\varepsilon}
\renewcommand{\rho}{\varrho}

\renewcommand{\phi}{\varphi}
\renewcommand{\hat}{\widehat}
\newcommand{\leer}{\varnothing}
\newcommand{\into}{\hookrightarrow}
\newcommand{\Pfeil}{\longrightarrow}
\newcommand\otau{\,{}^\tau\!}
\newcommand{\qH}{quasi-Hamiltonian}
\newcommand{\mf}{multiplicity free quasi-Hamiltonian}

\newcommand{\Times}{\mathop{\times}\limits}

\newcommand{\sfrac}[2]{{\textstyle\frac{#1}{#2}}}
\newcommand{\half}{\sfrac12}

\newcommand{\R}{\Sigma}

\DeclareMathOperator{\OG}{O}
\DeclareMathOperator{\Ad}{Ad}
\DeclareMathOperator{\Hom}{Hom}
\DeclareMathOperator{\Aut}{Aut}

\DeclareMathOperator{\res}{res}
\DeclareMathOperator{\Lie}{Lie}
\DeclareMathOperator{\rk}{rk}
\DeclareMathOperator{\pr}{pr}
\DeclareMathOperator{\id}{id}
\DeclareMathOperator{\Gr}{Gr}
\DeclareMathOperator{\Aff}{L}

\let\norm=\|
\def\|#1|{\operatorname{#1}}
\def\8{``}
\def\9{''}

\newcommand{\SL}{\|SL|}
\newcommand{\Sp}{\|Sp|}
\newcommand{\SO}{\|SO|}
\newcommand{\SU}{\|SU|}

\newcounter{rcount}
\newcommand{\Item}{\refstepcounter{rcount}\emph{\alph{rcount})\ }}

\title[]{Classification of multiplicity free quasi-Hamiltonian manifolds}

\author[]{Friedrich Knop (FAU Erlangen-Nürnberg)}
\address[]{Dept. Mathematik, FAU Erlangen-Nürnberg, Cauerstraße 11, 91058 Erlangen, Germany}

\dedicatory{To Corrado De Concini}

\usepackage{setspace}
\usepackage[colorlinks, breaklinks, allcolors=blue]{hyperref}
\usepackage[capitalize]{cleveref}

\usepackage{microtype}

\begin{document}

\setlength\mathindent{60pt}

\begin{abstract}

  A quasi-Hamiltonian manifold is called multiplicity free if all of
  its symplectic reductions are 0-dimensional. In this paper, we
  classify compact, multiplicity free, twisted \qH\ manifolds for
  simply connected, compact Lie groups. Thereby, we recover old and
  find new examples of these structures.

\end{abstract}

\maketitle

{\renewcommand{\baselinestretch}{0}\normalsize
\tableofcontents}

\section*{Introduction}\label{Intro}

Consider a compact, connected Lie group $K$. Within the class of
Hamiltonian $K$-manifolds, the most basic ones are known as
\emph{multiplicity-free} manifolds. These objects can be characterized
in several ways, such as the fact that their symplectic reductions are
zero-dimensional or that their generic orbits are coisotropic. In the
context of classical mechanics, multiplicity-free manifolds are
intimately connected with completely integrable systems.

In the 1980s, Delzant launched a program to classify compact,
multiplicity-free Hamiltonian manifolds. When $K$ is a torus that acts
effectively, he proved in \cite{Delzant0} that such a manifold $M$ is
uniquely determined by its momentum image $\cP_M$. Furthermore,
Delzant was able to provide a characterization of the sets that take
the form of $\cP_M$: they are precisely the \emph{simple polytopes},
i.e., polytopes which satisfy a technical integrality
condition. Delzant extended his investigation to non-abelian groups of
rank two in \cite{Delzant}, which led him to conjecture that, in
general, $M$ is uniquely determined by its momentum image $\cP_M$,
i.e., the image of the invariant momentum map (see Section
\ref{sec:localstructure}) and a certain lattice $\Lambda_M$ that
encodes the principal isotropy group.

Delzant's conjecture was subsequently confirmed in \cite{KnopAutoHam}
with an important step due to Losev \cite{Losev}. Furthermore, the
pairs $(\cP_M,\Lambda_M)$ arising this way were characterized in terms
of smooth affine, spherical varieties. This completed Delzant's
program.

Meanwhile, Alekseev, Malkin, and Meinrenken sought a way to describe
Hamiltonian actions for the \emph{loop group} of $K$, leading to the
development of the notion of a \emph{quasi-Hamiltonian manifold}
\cite{AMM}. These finite-dimensional $K$-manifolds are equipped with a
momentum map that takes values in $K$ instead of the dual Lie algebra
$\fk^*$.

In this paper, we extend the results in \cite{KnopAutoHam} to the
quasi-Hamiltonian case, by classifying compact, multiplicity-free
\qH\ manifolds under the condition that $K$ is simply connected. We
also consider the case where the momentum map is twisted by an
automorphism of $K$. Our result is that a compact \mf\ manifold $M$ is
classified by a pair $(\cP,\Lambda)$, which is compact and
\emph{spherical} with respect to an affine root system
(see \cref{def:spherical}).

To prove this, we follow the approach in \cite{KnopAutoHam} and first
study quasi-Hamiltonian manifolds locally over $\cP_M$. More
precisely, let $(\cP,\Lambda)$ be a fixed spherical pair. Then we show
that the category of \mf\ manifolds $M$ with $\cP_M\subseteq\cP$
open and $\Lambda_M=\Lambda$ forms a \emph{gerbe}. Crucial use is made
of a reduction procedure to ordinary Hamiltonian manifolds due to
Alekseev-Malkin-Meinrenken \cite{AMM} and Meinrenken
\cite{Meinrenken}.

Since all automorphism groups are abelian, they form a sheaf of
abelian groups over $\cP$, called a \emph{band}. To identify the band,
we use the computations in the Hamiltonian case \cite{KnopAutoHam}.

A central point is to show that the higher cohomology of this
band vanishes whenever $\cP$ is convex.  This cohomology computation is
more involved than in the Hamiltonian case due to complications with
affine root systems. General properties of gerbes then imply that
there exists precisely one $M$ with $(\cP_M,\Lambda_M)=(\cP,\Lambda)$.

We end the paper by constructing examples of compact \mf\ manifolds
using our classification. We recover some old examples, such as the
\emph{double of a group} by Alekseev-Malkin-Meinrenken \cite{AMM}, the
\emph{spinning $4$-sphere} by Alekseev-Meinrenken-Woodward \cite{AMW},
its generalization, the \emph{spinning $2n$-sphere} by
Hurtubise-Jeffreys-Sjamaar \cite{HJS}, and the \emph{quaternionic
  projective space} due to Eshmatov \cite{Eshmatov}. We end the paper,
by constructing some \mf\ manifolds which have not yet appeared in the
literature. For example we list all \mf\ $\SU(2)$-manifolds (twisted
and untwisted) and Eshmatov's example is extended to all quaternionic
Grassmannians. Most remarkable are \mf\ manifolds for which the
momentum map is surjective. These exist for example for $K=\SU(n)$ and
$K=\Sp(2n)$.

This paper replaces the preprint \cite{KnopQHam0} from 2016, which is
now obsolete. Our classification theorem has already been successfully
used by Paulus in his thesis \cite{Paulus} to classify other
interesting subclasses of \mf\ manifolds like those of with
one-dimensional momentum polytope, see
\cite{PaulusRk1,KnopPaulus}. Also the list of manifolds with
surjective momentum map has been completed, see
\cite{PaulusModel}.

\begin{ack}

  I would like to thank Chris Woodward who, a long time ago, suggested
  the topic of this paper to me. Thanks are also due to Kay Paulus and
  Bart Van Steirteghem for numerous discussions about this
  paper. Finally, I am indebted to the unknown referee for the many
  comments which highly improved the paper.

\end{ack}

\textbf{Notation:} \emph{a)} In the entire paper, $K$ will be a compact
connected Lie group. Whenever \qH\ manifolds are involved, $K$ will
additionally be assumed to be simply connected and equipped with an
automorphism $k\mapsto\otau k$ of $K$ (a \8twist\9) and a scalar
product on its Lie algebra $\fk$ which is invariant for both $K$ and
$\tau$.

\emph{b)} We adopt the following conventions: a \emph{polytope} is the
convex hull of a finite subset of a finite dimensional real affine
space while a \emph{polyhedron} is cut out by finitely many affine
linear inequalities $\alpha_1\ge0,\ldots,\alpha_n\ge0$. It is well
known, that polytopes are precisely the bounded polyhedra. A
\emph{polyhedral cone} is a subset of a real vector space which is
cut out by linear inequalities.

\emph{c)} As usual, we define the \emph{fiber product} $X\times_ZY$
with respect to two maps $\phi:X\to Z$ and $\psi:Y\to Z$ as the set of
$(x,y)\in X\times Y$ with $\phi(x)=\psi(y)$. The fiber product is
easily confused with the notation $K\times^LY$ for an \emph{associated
  fiber bundle}. The latter denotes by definition the orbit space
$(K\times Y)/L$ where $K$ is a group, $L$ is a subgroup, $Y$ is a set
with an $L$-action, and $L$ acts on $K\times Y$ via
$l\cdot(k,y):=(kl^{-1},ly)$.

\part{Local root systems and cohomology}

The principal purpose of this part is to state and prove the
Vanishing \cref{T1} which forms the central technical step in the
proof of the Classification \cref{thm:main}. In fact, the material of
sections \ref{sec:LRS}, \ref{sec:TAS}, and \ref{sec:TVT} is only used
in the proof of \cref{thm:main}.

\section{Affine root systems}\label{sec:affine}

In this section we set up notation for (affine) root systems. In
section~\ref{sec:twisted} they will be used to describe the geometry
of the (twisted) conjugation action of a compact connected Lie
group. They also control the automorphisms of \mf\ manifolds (see
\cref{thm:gerbe}). In the latter case, the root systems are more
complicated in that they don't have to be irreducible and that both
finite and affine root systems may occur as irreducible
components. Note also that our root systems carry a fixed metric as part
of the structure. Thus the approach is very similar to (and inspired
by) the treatises \cite{Mac1,Mac2} of Macdonald.

Let $\Vfa$ be an Euclidean vector space, i.e., a finite dimensional
$\RR$-vector space equipped with a positive definite scalar product
$\langle\cdot,\cdot\rangle$ and let $\fa$ be an affine space for
$\Vfa$, i.e., a non-empty set equipped with a free and transitive
$\Vfa$-action
\[
  \fa\times\Vfa\to\fa:(x,t)\mapsto x+t.
\]
The set of affine linear functions on $\fa$ is denoted by
$\Aff(\fa)$. Since $\Vfa$ carries a metric every $\alpha\in\Aff(\fa)$ has
a gradient $\Valpha:=\nabla\alpha\in\Vfa$. It is characterized by the
equation
\[
  \alpha(x+t)=\alpha(x)+\<\Valpha,t\>,\quad x\in\fa,t\in\Vfa.
\]
This way $\Aff(\fa)$ is an extension
\[
   0\to\RR\to\Aff(\fa)\overset\nabla\to\Vfa\to0.
\]
Every non-constant $\alpha\in\Aff(\fa)$ defines an affine hyperplane
$H_\alpha:=\{\alpha=0\}$ with normal vector $\Valpha$.

Similarly, let $M(\fa)$ be the group of isometries of $\fa$. For every
$w\in M(\fa)$ let $\Vw\in\OG(\Vfa)$ be its linear part. It is
characterized by the equation
\[
  w(x+t)=w(x)+\Vw(t),\quad x\in\fa,t\in\Vfa.
\]
This way, we obtain a short exact sequence
\[
  0\to\Vfa\to M(\fa)\to \OG(\Vfa)\to1.
\]
For a subgroup $W$ of $M(\fa)$ let $\VW$ be its image in $\OG(\Vfa)$.

A \emph{reflection} is an isometry $s\in M(\fa)$ whose fixed point set
$\fa^s$ is an affine hyperplane. Conversely, if $\alpha\in\Aff(\fa)$
is non-constant then there is a unique reflection $s_\alpha$ with
$\fa^{s_\alpha}=H_\alpha$ given by
\[
  s_\alpha(x)=x-\alpha(x)\,\Valpha^\vee,\quad x\in\fa\quad\text{ with
  }\quad \Valpha^\vee:=\frac2{\norm\Valpha\norm^2}\,{\Valpha}\in\Vfa.
\]
The induced action on $\Aff(\fa)$ is given by
\[
  s_\alpha(\beta)=\beta-\<\Vbeta,\Valpha^\vee\>\alpha,\quad \beta\in
  \Aff(\fa).
\]

After these preliminaries, affine root systems are defined as follows:

\begin{definition}
  A set $\Phi\subset \Aff(\fa)\setminus\RR$ of non-constant affine
  linear functions is an \emph{affine root system} if it has the
  following properties:

  \begin{enumerate}

  \item $\RR\alpha\cap\Phi=\{\alpha,-\alpha\}$ for all
    $\alpha\in\Phi$.

  \item $\<\Valpha,\Vbeta^\vee\>\in\ZZ$ for all $\alpha,\beta\in\Phi$.

  \item $s_\alpha(\Phi)=\Phi$ for all $\alpha\in\Phi$.

  \item $\VPhi:=\{\Valpha\in\Vfa\mid\alpha\in\Phi\}$ is finite.

  \end{enumerate}

  The \emph{Weyl group of $\Phi$} is the subgroup
  $W_\Phi\subseteq M(\fa)$ generated by all reflections $s_\alpha$,
  $\alpha\in\Phi$.

\end{definition}

This definition differs slightly from Macdonald's in two respects:
First, condition \emph{a)} says that we consider only reduced root
systems. Secondly, we do not assume $\Aff(\fa)$ to be spanned by
$\Phi$ or that $\Phi$ is infinite. In fact, $\Phi=\leer$ is also a
root system.

If $\Phi$ is an affine root system then for any $x\in\fa$
\[
  \Phi_x:=\{\alpha\in\Phi\mid\alpha(x)=0\}
\]
is a finite root system. Its Weyl group is the isotropy group of $x$
in $W_\Phi$ (see e.g. \cite{Bou}*{V, \S3, Prop.~1 (vii)}).

It is well known that every root system has a unique orthogonal
decomposition
\[
  \fa=\fa_0\times\fa_1\times\ldots\times\fa_n\quad\text{and}\quad
  \Phi=\Phi_1\cup\ldots\cup\Phi_n,
\]
such that the Weyl group
$W_\Phi=W_{\Phi_1}\times\ldots\times W_{\Phi_n}$ acts trivially on
$\fa_0$ and $\Phi_i\subset\Aff(\fa_\nu)$ is irreducible for
$\nu\ge 1$. Each pair $(\fa_\nu,\Phi_\nu)$ with $\nu\ge1$ corresponds
either to a finite, to an affine, or to a twisted affine Dynkin
diagram (see, e.g., \cite{Kac}*{Ch.\ 4}).

A {\it chamber of $\Phi$ (or $W_\Phi$)} is a connected component of
$\fa\setminus\bigcup_{\alpha\in\Phi}H_\alpha$. Its closure is called
an {\it alcove}. It is known, that $W_\Phi$ acts simply transitively
on the set of alcoves and that each alcove $\cA$ is a fundamental
domain for $W_\Phi$. The reflections about the walls of $\cA$ are
called \emph{simple} (with respect to $\cA$). They generate $W_\Phi$
as a group.

If the factor $\Phi_\nu$, $\nu\ge1$, is finite then its alcoves are
simplicial cones. Otherwise, they are simplices. So in general, an
alcove is a product of an affine space, a simplicial cone and a finite
number of simplices.

The set $\VPhi\subseteq\Vfa$ of gradients of affine roots is a finite,
but possibly non-reduced root system (e.g., if the root system is of
type $\sA_{2n}^ {(2)}$). Its Weyl group $W_\VPhi$ is the image
$\VW_\Phi$ of $W_\Phi$ in $\OG(\Vfa)$.

If $\Lambda\subseteq\Vfa$ is a lattice we will denote its dual lattice
$\{\chi\in\Vfa\mid\<\Lambda,\chi\>\subseteq\ZZ\}$ by $\Lambda^\vee$.

\begin{definition}\label{def:weight}

  Let $\Phi\subset \Aff(\fa)$ be an affine root system. A
  \emph{weight lattice for $\Phi$} is a lattice $\Lambda\subseteq\Vfa$
  with
  \[
    \VPhi\subset\Lambda\text{ and }\VPhi^\vee\subset\Lambda^\vee.
  \]
  The pair $(\Phi,\Lambda)$ is called an \emph{integral affine root
    system}.

\end{definition}

Let $(\Phi,\Lambda)$ be an integral affine root system on $\fa$. Then
its Weyl group will also act on the compact torus
$A:=\Vfa/\Lambda^\vee$. The character group $\Xi(A)$ can be identified
with $\Lambda$. More specifically, to $\chi\in\Lambda$ we attach the
character
\[\label{eq:chitilde}
  \tilde\chi(a+\Lambda^\vee):=e^{2\pi i\,\<\chi,a\>}.
\]
For ease of notation, we are going to set
$\tilde\alpha:=\widetilde{\Valpha}$ for $\alpha\in\Phi$. Dually, every
$\eta\in\Lambda^\vee$ defines the cocharacter
\[
  \tilde\eta:U(1)\to A:e^{2\pi i\,t}\mapsto t\eta +\Lambda^\vee.
\]
Again, for $\alpha\in\Phi$ we write
$\tilde\alpha^\vee:=\widetilde{\Valpha^\vee}$. Then we have the
formula
\[
  \tilde\chi(\tilde\alpha^\vee(u))=u^{\<\chi,\Valpha^\vee\>}\text{ for
    all }\chi\in\Lambda, \alpha\in\Phi, u\in U(1).
\]
Note that this implies in particular
\[\label{e1}
  \tilde\alpha(\tilde\alpha^\vee(u))=u^2\quad\text{ for all }\alpha\in\Phi,
  u\in U(1).
\]
The action of a reflection $s_\alpha\in W_\Phi$ is given by the
formula
\[\label{e2}
  s_\alpha(a)=a\cdot\tilde\alpha^\vee(\tilde\alpha(a))^{-1}\quad\text{ for
    all }a\in A.
\]

\section{Local root systems}\label{sec:LRS}

In this section we introduce a localized version of a root system.

\begin{definition}\label{def:LRS}
  Let $\cP$ be a subset of the affine space $\fa$. A \emph{local root
    system on $\cP$} is a family $\Phi(*)=(\Phi(x))_{x\in\cP}$ and a
  lattice $\Lambda\subseteq\Vfa$ with:
  
  \begin{enumerate}
  \item\label{it:localRoot0} For each $x\in\cP$, the pair
    $(\Phi(x),\Lambda)$ is an integral affine root system on $\fa$.
    
  \item\label{it:localRoot1} Every $x\in\cP$ has a neighborhood
    $U\subseteq\cP$ such that $\Phi(y)=\Phi(x)_y$ for all $y\in U$.

  \item\label{it:localRoot2} Every $\alpha\in\Phi(x)$, 
    $x\in\cP$, has $\alpha|_\cP\ge0$ or $\alpha|_\cP\le0$.

  \end{enumerate}
\end{definition}

Observe that \ref{it:localRoot1} applied to $x=y$ means $\alpha(x)=0$
for all $\alpha\in\Phi(x)$. Hence each of the root systems $\Phi(x)$
is finite.

If $(\Phi,\Lambda)$ is an integral affine root system and $\cP$ is a subset
of an alcove then the pair $((\Phi_x)_{x\in\cP},\Lambda)$ is a local
root system. Systems of this type are called \emph{trivial}.

\begin{definition}
  
  A non-empty subset $\cP\subseteq\fa$ is called \emph{solid} if its
  \emph{interior} $\cP^0$ (i.e., the largest open subset contained in
  $\cP$) is dense in $\cP$. Observe that if $\cP$ is convex then $\cP$
  is solid if and only if $\dim\cP=\dim\fa$ if and only if $\cP$ spans
  $\fa$ as an affine space.

\end{definition}

The goal of this section is to prove the following triviality
criterion.

\begin{proposition}\label{prop:TrivialRootCrit}
  Let $(\Phi(*),\Lambda)$ be a local root system on $\cP\subseteq\fa$
  and let $W\subseteq M(\fa)$ be the subgroup generated by all local
  Weyl groups $W(x)$ of $\Phi(x)$ with $x\in\cP$. Assume:

  \begin{enumerate}
  \item The subset $\cP$ is convex and solid.

  \item\label{it:TrivialRootCrit-2} Every $W$-orbit meets $\cP$ in at most one
    point.
  \end{enumerate}

  Then the local root system is trivial.
\end{proposition}

\begin{remark}
  In the application later, $\fa$ will be a subspace of an affine
  space $\ft$ which carries an affine reflection group $\tilde W$ such
  that each element of $W(x)$ is induced by an element of $\tilde W$
  and such that $\cP$ lies in an alcove of $\tilde W$. In this
  setting, condition \ref{it:TrivialRootCrit-2} in
  \cref{prop:TrivialRootCrit} holds since it already holds for
  $\tilde W$-orbits in $\fa$.
\end{remark}

The proof will proceed in two steps. First we consider the
corresponding system of Weyl groups $(W(x))_{x\in\cP}$ and prove that
it is trivial under similar assumptions. From that we deduce that the
local root system itself is trivial.

The reflection group analogue for \cref{def:LRS} is:

\begin{definition}
  A \emph{local reflection group on a subset $\cP\subseteq\fa$} is a
  family $W(*)=(W(x))_{x\in\cP}$ with the following properties:
  
  \begin{enumerate}
  \item\label{it:localRef0} $W(x)$ is a reflection group on $\fa$ for
    each $x\in\cP$.
    
  \item\label{it:localRef1} Every $x\in\cP$ has a neighborhood
    $U\subseteq\cP$ such that $W(y)=W(x)_y$ for all $y\in U$ where
    $W(x)_y$ is the isotropy group of $y$ inside $W(x)$.

  \item\label{it:localRef2} For every reflection $s\in W(x)$,
    $x\in\cP$, the set $\cP$ lies entirely in one of the two closed
    halfspaces determined by the reflection hyperplane $\fa^s$.

  \end{enumerate}
\end{definition}

Again condition \ref{it:localRef1} implies that $x$ is a fixed point of
$W(x)$ and therefore that $W(x)$ is a finite reflection group.

If $\Phi(*)$ is a local root system then its system of Weyl groups
$(W(x))_{x\in\cP}$ forms a local reflection group. Moreover, if $W$ is
an affine reflection group on $\fa$ and $\cP$ a subset of an alcove
then the family of isotropy groups $(W_x)_{x\in\cP}$ is a local
reflection group on $\cP$. These local reflection groups will be
called \emph{trivial}.

The main tool for showing triviality is the following classical
criterion for a given set of reflections to be the set of simple
reflections of an affine reflection group.

\begin{lemma}\label{L1}

  Let $\alpha_1,\ldots,\alpha_n\in\Aff(\fa)$ be non-constant affine
  linear functions with:

  \begin{enumerate}

  \item\label{L1i1} For any $i\ne j$, the angle between $\Valpha_i$
    and $\Valpha_j$ equals $\pi-\frac\pi\ell$ with
    $\ell\in\ZZ_{\ge2}\cup\{\infty\}$.

  \item\label{L1i2} There is a point $x\in\fa$ with $\alpha_i(x)>0$
    for all $i=1,\ldots,n$.

  \end{enumerate}

\noindent
Let $W\subseteq M(\fa)$ be the group generated by the reflections
$s_{\alpha_1},\ldots,s_{\alpha_n}$. Then $W$ is an affine
reflection group,
\[
  \cA:=\{x\in\fa\mid \alpha_1(x)\ge0,\ldots,\alpha_n(x)\ge0\}
\]
is an alcove for $W$, and the reflections
$s_{\alpha_1},\ldots,s_{\alpha_n}$ are the simple reflection of $W$
with respect to $\cA$.

\end{lemma}

\begin{proof}

  Condition \ref{L1i2} implies that $\cA$ is a solid
  polyhedron. Let, after renumbering, $\alpha_1,\ldots,\alpha_m$ be
  the non-redundant functions defining $\cA$, i.e., those whose
  intersection $\{\alpha_i=0\}\cap\cA$ is of codimension $1$ in
  $\cA$. Then a classical theorem (see e.g. Vinberg
  \cite{Vin}*{Thm.~1} for a much more general statement) asserts that,
  under condition \ref{L1i1}, $s_{\alpha_1},\ldots,s_{\alpha_m}$ are
  the simple reflections for an affine reflection group $W$ and that
  $\cA$ is a fundamental domain. So, it remains to show that
  $m=n$. Suppose not. Then $\alpha_{m+1}$ would be redundant. This
  implies that there are real numbers $c_1,\ldots,c_m\ge0$ such that
  $\alpha_{m+1}=\sum_{i=1}^mc_i\alpha_i$. From \ref{L1i1} we get that
  \[
    \<\Valpha_i,\Valpha_{m+1}\>=
    \norm\Valpha_i\norm\norm\Valpha_{m+1}\norm\cdot\cos(\pi-\frac\pi\ell)\le0
  \]
  for $i=1,\ldots,m$ and therefore the contradiction
  $\<\Valpha_{m+1},\Valpha_{m+1}\>\le0$.
\end{proof}

The triviality criterion for local reflection groups is:

\begin{lemma}\label{L2}

  Let $(W(x))_{x\in\cP}$ be a local reflection group on
  $\cP\subseteq\fa$. Let $W$ be the group generated by all $W(x)$,
  $x\in\cP$. Assume:

  \begin{enumerate}
  \item $\cP$ is convex and solid.

  \item Every $W$-orbit in $\fa$ meets $\cP$ in at most one point.

  \end{enumerate}
  
  Then $W$ is an affine reflection group with $W(x)=W_x$ for all
  $x\in\cP$.

\end{lemma}

\begin{proof}

  Let me first remark that it is important to keep in mind that the
  various reflection hyperplanes might not meet within $\cP$. Typical
  is the situation of figure \eqref{eq:fig1} where the shaded area is
  $\cP$ and the local Weyl group at each vertex is generated by the
  reflection about the dashed lines through the vertex.
  
  In a first step we claim that
  \[\label{eq:e1}
    W(x)_y=W(y)_x\quad\text{for all $x,y\in\cP$}.
  \]
  Indeed, let $l=[x,y]\subseteq\fa$ be the line segment joining $x$
  and $y$. Then $l\subseteq\cP$ since $\cP$ is convex. For any
  $z\in l$ let
  \[
    W(z)_l:=\{w\in W(z)\mid wu=u\text{ for all $u\in l$}\}
  \]
  Then
  \[
    W(u)_l=(W(z)_u)_l=W(z)_l
  \]
  for all $u\in l$ which are sufficiently close to $z$. This means
  that the map $z\mapsto W(z)_l$ is locally constant, hence constant,
  on $l$. Thus
  \[
    W(x)_y=W(x)_l=W(y)_l=W(y)_x.
  \]

  Let $s=s_\alpha\in W$ be a reflection with fixed point set
  $H:=\{\alpha=0\}$, $i=1,2$. We claim that $H$ does not meet $\cP^0$,
  the interior of $\cP$ inside $\fa$. Otherwise, there would
  be points $x,y\in\cP^0$ with $\alpha(x)>0$ and $\alpha(y)<0$. The
  line segment joining $x$ and $y$ lies entirely in $\cP^0$ and meets
  $H$ in exactly one point $z$. Moreover there is an $\epsilon>0$ such
  that both points $z_\pm:=z\pm\epsilon\Valpha$ are in $\cP^0$. But
  then $z_+$ and $z_-=s(z_+)$ would be two different points of $\cP$
  lying in the same $W$-orbit contradicting our assumption.

  The claim implies that $\cP^0$, being connected, lies entirely in
  one of the open halfspaces determined by $H$. Hence $\cP$ lies
  entirely in one of the two closed halfspaces determined by $H$.

  This reasoning applies, in particular, to all reflections contained
  in $W(x)$, where $x\in\cP$. Thus, $\cP$ is contained in a unique
  closed Weyl chamber $C(x)\subseteq\fa$ for $W(x)$. This chamber
  determines in turn a set $\R(x)\subset W(x)$ of simple
  reflections. It is well-known that for every $y\in C(x)$ the set
  $\R(x)_y:=\{s\in \R(x)\mid sy=y\}$ is a set of simple reflections for
  $W(x)_y$. Therefore equation \eqref{eq:e1} implies that
  \[
    \R(x)_y=\R(y)_x\quad\text{for all $x,y\in\cP$}.
  \]
  Now let $\R$ be the union of all $\R(x)$, $x\in\cP$. Then
  \[\label{eq:Sx}
    \R(x)=\{s\in \R\mid sx=x\}
  \]
  for all $x\in\cP$. Indeed, let $s\in \R$ with $sx=x$. Then
  $s\in \R(y)$ for some $y\in\cP$. Thus,
  $s\in \R(y)_x=\R(x)_y\subseteq \R(x)$.

  For each $s\in \R$ choose affine linear functions $\alpha_s$ with
  $s=s_{\alpha_s}$ and such that $\alpha_s\ge0$ on $\cP$. We are going
  to show that $\{\alpha_s\mid s\in \R\}$ satisfies the assumptions of
  Lemma \ref{L1}.

  Let $s_1\ne s_2\in \R$. Put $\alpha_i:=\alpha_{s_i}$ and
  $H_i:=\{\alpha_i=0\}$. Assume first that $H_1$ and $H_2$ are
  parallel. Then $\Valpha_1=c\Valpha_2$ with $c\ne0$ and we have to
  show that $c<0$. The functions $\alpha_i$ vanish, by construction,
  at some points $x_i\in\cP$. Put $t:=x_1-x_2\in\Vfa$. Then
  $\<\Valpha_1,t\>=-\alpha_1(x_2)<0$ and
  $\<\Valpha_2,t\>=\alpha_2(x_1)>0$ which shows $c<0$.

  Now assume that $H_1$ and $H_2$ are not parallel. Then
  $E:=H_1\cap H_2$ is a subspace of codimension two. Let
  $W'\subseteq W$ be the dihedral group generated by $s_1$ and $s_2$
  and let $\theta$ be the angle between $\Valpha_1$ and
  $\Valpha_2$. Then $W'$ contains the rotation $r$ around $E$ with
  angle $2\theta$. If $r$ had infinite order then the union of all
  $\<r\>$-translates of, say, $H_1$ would be dense in $\fa$. Since
  $\cP$ is solid that contradicts the assumption that every $W$-orbit
  meets $\cP$ at most once. Therefore $W'$ is a finite reflection
  group.

  Now we claim that $\{s_1,s_2\}$ is a set of simple reflections for
  $W'$. If $E\cap\cP\ne\leer$ this is clear since
  $s_1,s_2\in \R(x)$ for all $x\in E\cap\cP$ (by
  eqn.~\eqref{eq:Sx}). So assume $E\cap\cP=\leer$. Let $C'$ be the
  unique Weyl chamber of $W'$ which contains $\cP$ and let
  $s_i'\in W'$, $i=1,2$, be the corresponding simple
  reflections. Choose functions $\alpha_i'$ with $s_i'=s_{\alpha_i'}$
  such that $\alpha_i'\ge0$ on $\cP$. Observe that
  \[
    E=\{\alpha_1=\alpha_2=0\}=\{\alpha_1'=\alpha_2'=0\}=\fa^{W'}.
  \]
  Now fix $i\in\{1,2\}$. Then $\alpha_i=c_1\alpha_1'+c_2\alpha_2'$ for
  some real numbers $c_1,c_2\ge0$. Suppose $c_1,c_2>0$, i.e., $s_i$ is
  not simple. By construction $s_i\in W(x)$ for some $x\in\cP$. Then
  \[
    0=\alpha_i(x)=c_1\alpha_1'(x)+c_2\alpha_2'(x)
  \]
  implies $\alpha_1'(x)=\alpha_2'(x)=0$ and therefore $x\in\cP\cap E$
  which is excluded.

  The fact that $s_1$ and $s_2$ are simple reflections of $W'$ implies
  that the angle between $\Valpha_1$ and $\Valpha_2$ is of the form
  $\pi-\frac\pi\ell$ with $\ell\in\ZZ_{\ge2}$.  Since condition
  \ref{L1i2} of \cref{L1} is obvious from $\cP^0\subseteq\cA$ and the
  fact that $\cP$ is solid we can apply \cref{L1} and infer that $W$
  is an affine reflection group with alcove $\cA$ containing $\cP$ and
  that $\R$ is a set of simple reflections of $W$. Finally,
  \eqref{eq:Sx} implies
  \[
    W_x=\langle s\in \R\mid sx=x\rangle=\langle \R(x)\rangle=W(x).
  \]
  for all $x\in\cP$.
\end{proof}

For the second step of the proof of \cref{prop:TrivialRootCrit} we
analyze to what extent a root system $\Phi$ is determined by its Weyl
group $W$ and a weight lattice $\Lambda$.

Choose an alcove $\cA$ of $W$ and let $\R\subseteq W$ be the set of
simple reflections with respect to $\cA$. For every $s\in\R$ there is
a unique affine linear function $\pi_s\in\Aff(\fa)$ such that
$\{\pi_s=0\}=\fa^s$, $\pi_s|_\cA\ge0$, and $\overline\pi_s\in\Lambda$
is primitive.

Let $(\Phi,\Lambda)$ be an integral affine root system with Weyl group $W$
and let $\alpha_s\in\Phi$ be the simple root corresponding to
$s$. Then $\alpha_s=n\pi_s$ with $n\in\ZZ_{>0}$. Since
$\Valpha_s^\vee=\frac 1n\Vpi_s^\vee\in\Lambda^\vee$ we have
$\<\chi,\Valpha_s^\vee\>=\frac1n\<\chi,\Vpi_s^\vee\>\in\ZZ$. Applied to
$\chi=\pi_s$ one gets $n=1$ or $n=2$. Moreover, if $n=2$ then
$\<\Lambda,\Vpi_s^\vee\>=2\ZZ$. Therefore, we define the set of
\emph{ambiguous reflections} as
\[
  \R^a:=\R^a(\Lambda)=\{s\in\R\mid\<\Lambda,\Vpi_s^\vee\>=2\ZZ\}
\]
and
\[
  \R^a(\Phi):=\R^a(\Phi,\Lambda):=\{s\in\R\mid\alpha_s=2\pi_s\}.
\]
Then the discussion above implies that $\R^a(\Phi)\subseteq\R^a$ and
that the set of simple roots of $\Phi$ and therefore $\Phi$ itself is
determined by $\R^a(\Phi)$. To see that every subset of $\R^a$ can be
realized this way we need the following lemma

\begin{lemma}\label{lemma:conj}
  Let $s\in\R^a$ and $t\in\R$ be $W$-conjugate. Then $s=t$.
\end{lemma}

\begin{proof}
  Since $s,t\in\R$ are conjugate there is a string of simple
  reflections $s=s_1,s_2,\ldots,s_n=t$ such that the order of
  $s_\nu s_{\nu+1}$ is odd for all $\nu=1,\ldots,n-1$ (see
  e.g. \cite{Bou}*{IV, \S1, Prop. 3}). For Weyl groups this happens only
  if $\pi_{s_\nu}$, $\pi_{s_{\nu+1}}$ span a root system type
  $\sA_2$. Thus, if $s\ne t$ then $n\ge2$ and
  $\<\pi_{s_2},\Vpi_{s_1}^\vee\>=-1\not\in2\ZZ$ in contradiction to $s$
  being ambiguous.
\end{proof}

\begin{lemma}\label{L3}

  Fix an affine reflection group $W$ on $\fa$, a $W$-invariant lattice
  $\Lambda\subseteq\Vfa$, and an alcove $\cA$ of $W$. Then the map
  $\Phi\mapsto \R^a(\Phi)$ is a bijection between integral affine root
  systems $(\Phi,\Lambda)$ with $W_\Phi=W$ and subsets of $\R^a$.

\end{lemma}

\begin{proof}
  It remains to prove that for every $I\subseteq\R^a$ there is a root
  system $\Phi_I$ with $\R^a(\Phi_I)=I$. If $\Phi_I$ exists at all
  then its set $S_I$ of simple roots has to be
  $\{\alpha_s\mid s\in\R\}$ with
   \[
    \alpha_s:=\begin{cases} 2\pi_s&\text{if }s\in I,\\
      \pi_s&\text{if }s \in \R\setminus I.
    \end{cases}
  \]
  Then $\Phi_I=WS_I$ is a root system except that it might not be
  reduced. So suppose $\alpha,\beta\in\Phi_I$ are distinct and
  positively proportional. By applying an element of $W$ and by the
  discussion above \cref{lemma:conj} we may assume without loss of
  generality that $\alpha=\pi_s$ and $\beta=2\alpha_s$ for some
  $s\in\R$. In particular $s\in\R^a$. On the other hand there are
  $t\in\R$ and $w\in W$ with $\beta=w\alpha_t$. Because of
  $s=s_\alpha=s_\beta=wtw^{-1}$ it follows from \cref{lemma:conj} that
  $s=t$ and therefore $\alpha_s=\alpha_t$. But that is impossible
  since $\alpha$ is primitive and $\beta$ is not.
\end{proof}

  Now we can finish the

\begin{proof}[Proof of \cref{prop:TrivialRootCrit}]
  The system $W(x)$ of Weyl groups forms a local reflection group on
  $\cP$. Moreover, the assumptions on $\Phi(*)$ imply the assumptions
  of \cref{L2}. Therefore, $W$ is an affine reflection group with
  $W(x)=W_x$ for all $x\in\cP$. In particular, $(\Phi(x),\Lambda)$ is
  an integral affine root system with Weyl group $W_x$. Let
  $\R\subseteq W$ be the set of simple reflections. Then
  $\R_x=\R\cap W_x$ is a set of simple reflections for $W_x$. The
  integral affine root system $(\Phi(x),\Lambda)$ is therefore
  determined by a subset $\R^a(x):=\R^a(\Phi(x))\subseteq\R^a$. Now
  the same argument as for \eqref{eq:e1} also shows
  \[\label{eq:phiphi}
    \Phi(x)_y=\Phi(y)_x\text{ for all $x,y\in\cP$}.
  \]
  This implies that whenever $s\in\R_x\cap\R_y$ then $s\in\R^a(x)$ if
  and only if $s\in\R^a(y)$. Thus, the union
  $\R^a(*):=\bigcup_{x\in\cP}\R^a(x)$ has the property that
  $\R^a(*)\cap\R_x=\R^a(x)$ for all $x\in\cP$. Let $\Phi$ be the root
  system with $\R^a(\Phi)=\R^a(*)$ whose existence is guaranteed by
  \cref{L3}. Because of
  $\R^a(\Phi_x)=\R^a(\Phi)\cap\R_x=\R^a(x)=\R^a(\Phi(x))$ we obtain
  $\Phi_x=\Phi(x)$, as required.
\end{proof}

\section{The automorphism sheaf}\label{sec:TAS}

We keep the notation of section \ref{sec:affine}: $(\Phi,\Lambda)$ is
an integral affine root system on the affine space $\fa$ with Weyl
group $W$ and fundamental alcove $\cA$.  Recall the torus
$A:=\Vfa/\Lambda^\vee$. Let, moreover, $\cP\subseteq\cA$ be a solid
subset which is additionally assumed to be \emph{locally polyhedral},
i.e., every $x\in\cP$ has a neighborhood $U\subseteq\fa$ with
$\cP\cap U=\cQ\cap U$ for some polyhedron $\cQ\subseteq\fa$ depending
on $x$.

In this section, we consider certain maps $\phi:\cP\to A$.

\begin{enumerate}

\item A map $\phi:\cP\rightarrow A$ is \emph{smooth} if every point of
  $\cP$ has an open neighborhood $U\subseteq\fa$ such that $\phi$ is
  the restriction of a smooth map $\tilde\phi:U\to A$ to
  $U\cap\cP$. Let $\hat\cC_{\fa,x}$ and $\hat\cC_{A,a}$ be the
  completions of the local ring of smooth functions (i.e., the rings
  of formal power series) in $x\in\fa$ and $a\in A$,
  respectively. Then a smooth map $\phi$ with $a=\phi(x)$ induces an
  algebra homomorphism
  \[\label{eq:hatC}
    \hat\phi_x:\hat\cC_{A,a}\to\hat\cC_{\fa,x}.
  \]
  In fact, any local extension $\tilde\phi:U\to A$ defines a
  homomorphism $\hat{\tilde\phi}_x:\hat\cC_{A,a}\to\hat\cC_{\fa,x}$
  which, by continuity, is independent of the choice of $\tilde\phi$
  since $\cP^0$ is dense in $\cP$.  Thus
  $\hat\phi_x:=\hat{\tilde\phi}_x$ is well defined.

\item Since each tangent space of the product space $\fa\times A$ is
  canonically isomorphic to $\Vfa\oplus\Vfa$, the scalar product on
  $\Vfa$ induces a symplectic structure on
  $\fa\times A$ by
  \[
    \omega(\xi_1+\eta_1,\xi_2+\eta_2)=\<\xi_1,\eta_2\>-\<\xi_2,\eta_1\>,
    \quad\text{with }\xi_1,\xi_2,\eta_1,\eta_2\in\Vfa.
  \]
  Using the identifications $\fa\cong\Vfa\cong\Vfa^*$ one can consider
  $\omega$ as the canonical symplectic form on the cotangent bundle
  $T^*_A=\Vfa^*\times A$. A smooth map $\phi:\cP\to A$ is
  \emph{closed} if the graph of $\phi|_{\cP^0}$ is a Lagrangian
  submanifold of $\fa\times A$.

\item A smooth map $\phi:\cP\to A$ is \emph{$W$-equivariant} if for
  every $x\in\cP$ the point $a=\phi(x)$ is $W_x$-fixed (i.e.,
  $w\in W,wx=x\Rightarrow wa=a$) and the induced homomorphism
  \eqref{eq:hatC} is $W_x$-equivariant.

\item A smooth map $\phi:\cP\rightarrow A$ is
  \emph{$\Phi$-equivariant} if it is $W$-equivariant and
  \[\label{eq:atpx}
    \tilde\alpha(\phi(x))=1\text{ for all $x\in \cP$ and all roots
      $\alpha\in\Phi$ with $\alpha(x)=0$,}
  \]
  where $\tilde\alpha$ is as in \eqref{eq:chitilde}.

\end{enumerate}

\begin{Remarks}\label{Rem:auto}
  
  \emph{a)} If $\cP$ is convex, so in particular simply connected, the
  notion of closedness can be rephrased: Because $\exp:\Vfa\to A$ is a
  covering the map $\phi$ can be lifted to a smooth map
  $\tilde\phi:\cP\to\Vfa$. Because of the identification
  $\Vfa\cong\Vfa^*$ one can interpret $\tilde\phi$ as a $1$-form. Then
  it easy to see that $\phi$ is closed if and only if $\tilde\phi$ is
  closed as a $1$-form (whence the name).

\emph{b)} For $\alpha\in\Phi$ let $s_\alpha\in W$ be the corresponding
  reflection. Then $s_\alpha\in W_x$ if and only if $\alpha(x)=0$. In
  this case, $W$-equivariance implies $s_\alpha(a)=a$ where
  $a=\phi(x)$. This means
  \[\label{eq:1A}
    \tilde\alpha^\vee(\tilde\alpha(a))=1\in A
  \]
  by equation \eqref{e2}. Applying $\tilde\alpha$ to both sides, we
  see (equation \eqref{e1}) that $W$-equivariance alone already
  implies
  \[\label{eq:alphaquadrat}
    \tilde\alpha(a)^2=1\in\RR\text{, i.e., }\tilde\alpha(a)=\pm1.
  \]
  So $\Phi$-equivariance just means that additionally
  $\tilde\alpha(a)$ equals $1$ instead of $-1$.

\end{Remarks}

Now we localize these definitions.

\begin{definition}\label{def:fL}

  Let $((\Phi(x))_{x\in\cP},\Lambda)$ be a local system of roots on
  $\cP$ and $U\subseteq\cP$ open.  Then
  $\fL^{\Phi(*)}_{\cP,\Lambda}(U)$ is the set of smooth and closed
  maps $U\to A$ which are $\Phi(x)$-equivariant for every $x\in U$.

\end{definition}

Clearly $\fL^{\Phi(*)}_{\cP,\Lambda}$ is a sheaf of abelian groups on
$\cP$. If the local root system is trivial and comes from an integral
affine root system $(\Phi,\Lambda)$ we simply write
$\fL^\Phi_{\cP,\Lambda}$.

\section{The vanishing theorem}\label{sec:TVT}

In this section, we compute the cohomology of
$\fL^\Phi_{\cP,\Lambda}$. This will be the central step in the proof
of \cref{thm:main}.

\begin{theorem}\label{T1}

  Let $(\Phi,\Lambda)$ be an integral affine root system on the affine
  space $\fa$, let $\cA\subseteq\fa$ be an alcove and let
  $\cP\subseteq\cA$ be a convex, solid, locally polyhedral
  subset. Then $H^i(\cP,\fL^\Phi_{\cP,\Lambda})=0$ for all $i\ge1$.

\end{theorem}

The proof will occupy the rest of this section. We start with a
reduction step:

\begin{lemma}\label{lem:commensurable}

  Let $\Lambda_1,\Lambda_2\subseteq\Vfa$ be two commensurable weight
  lattices for $\Phi$ (i.e. $\Lambda_1\cap\Lambda_2$ is also a
  lattice). Then
  $H^i(\cP,\fL^\Phi_{\cP,\Lambda_1})=H^i(\cP,\fL^\Phi_{\cP,\Lambda_2})$
  for all $i\ge1$

\end{lemma}

\begin{proof}

  By replacing $\Lambda_1$ with the intersection
  $\Lambda_1\cap\Lambda_2$ we may assume
  $\Lambda_1\subseteq\Lambda_2$. Then $A_1:=\Vfa/\Lambda_1^\vee$ is a
  quotient of $A_2:=\Vfa/\Lambda_2^\vee$ with kernel
  \[\label{e3}
    E:=\Lambda_1^\vee/\Lambda_2^\vee\subseteq A_2.
  \]
  Let $U\subseteq\cP$ be convex and open. Then any smooth map
  $\phi_1:U\to A_1$ can be lifted to a smooth map $\phi_2:U\to
  A_2$. This lifted map $\phi_2$ is closed and $\Phi$-equivariant if and
  only $\phi_1$ is. Thus, we get a short exact sequence of sheaves
  \[\label{e4}
    0\Pfeil
    E_\cP\to\fL^\Phi_{\cP,\Lambda_2}\Pfeil\fL^\Phi_{\cP,\Lambda_1}\Pfeil0
  \]
  where $E_\cP$ denotes the constant sheaf on $\cP$ with fiber
  $E$. Since $\cP$ is convex, we have $H^i(\cP,E_\cP)=0$ for $i\ge1$
  and the assertion follows.
\end{proof}

A weight lattice will be called \emph{of adjoint type} if
\[
  \Lambda=\ZZ\VPhi\oplus\Lambda^W\subseteq\RR\VPhi\oplus\Vfa^W=\Vfa.
\]
Since every weight lattice $\Lambda$ is commensurable to
$\ZZ\VPhi\oplus\Lambda^W$ \cref{lem:commensurable} allows to assume
that $\Lambda$ is of adjoint type.

Next, we need a method to produce sections of
$\fL^\Phi_{\cP,\Lambda}$. To this end, we define a function $f$ on
$\cP$ to be \emph{smooth} if it can be locally extended to a smooth
function on an open subset of $\fa$. The differential $df$ of a smooth
function $f$ can be considered as a smooth map $\cP\to\Vfa^*$. Since
the form $df$ is closed it follows from Remark~\ref{Rem:auto}
\emph{a)} that $\phi:=\exp(\nabla f)$ defines a closed smooth
map $\cP\to A$. After localizing this construction, we get
a homomorphism of sheaves
\[\label{eq:defepsilon}
  \epsilon:\cC_\cP\to\fL^{\text{closed}}_{\cP,\Lambda}
\]
where $\cC_{\cP}$ is the sheaf of smooth functions on $\cP$ and
$\fL^{\text{closed}}_{\cP,\Lambda}$ is the sheaf of closed maps from $\cP$ to
$A$.

\begin{lemma}\label{lemma:nabla}
  The homomorphism $\epsilon$ is surjective, i.e., all closed
  maps $\phi$ from $\cP$ to $A$ are locally of the form
  \[\label{eq:nabla}
    \phi(x)=\exp(\nabla f(x))
  \]
  where $f$ is a smooth function on an open subset of $\cP$.
\end{lemma}

\begin{proof}
  Let $U\subseteq\fa$ be a convex open neighborhood of $a\in\cP$ such
  that $\cP\cap U$ is convex and $\phi$ is defined on $\cP\cap
  U$. Then $\phi|_{\cP\cap U}$ lifts to a closed $1$-form
  $\tilde\phi:\cP\cap U\to\Vfa\overset\sim\to\Vfa^*$ which extends to
  a smooth $1$-form $\omega$ on $U$. The derivative $d\omega$ of this
  form vanishes on $\cP\cap U$. The convexity of $U$ allows us to
  define the smooth function $f(x):=\int_{[a,x]}\omega$ in $U$ where
  $[a,x]$ is the line segment from $a$ to $x$. Because $\cP\cap U$ is
  convex this line segment lies entirely in $\cP\cap U$ when
  $x\in\cP\cap U$. Since $d\omega$ vanishes on $\cP\cap U$ (the proof
  of) Poincaré's Lemma shows $df|_{\cP\cap U}=\tilde\phi$.
\end{proof}

This construction produces closed maps to $A$. To get $W$-equivariant
ones let $f$ be a \emph{$W$-invariant} smooth function in the sense
that for each $x\in\cP$ the Taylor series of $f$ in $x$ is
$W_x$-invariant. Then $\epsilon(f)$ is a $W$-equivariant closed map to
$A$. We claim that $\epsilon(f)$ is automatically
$\Phi$-equivariant. Indeed, consider the continuous family
$\phi_t=\epsilon(tf)$, $t\in\RR$ of closed maps and let
$\alpha\in\Phi$ with $\alpha(x)=0$. Since
$\tilde\alpha(\phi_t(x))\in\{\pm1\}$ (see \eqref{eq:alphaquadrat}) and
$\phi_0(x)=1$ we get $\tilde\alpha(\phi_1(x))=1$ by continuity.

Thus, if we denote the sheaf of $W$-invariant smooth
functions on $\cP$ by $\cC_\cP^W$ we obtain a homomorphism of sheaves
\[\label{eq:epsilon}
  \epsilon^W:\cC_\cP^W\to\fL^\Phi_{\cP,\Lambda}.
\]
The first step towards proving the vanishing theorem is:

\begin{lemma}\label{lemma:soft}
  $H^i(\cP,\cC_\cP^W)=0$ for $i\ge1$.
\end{lemma}

\begin{proof}
  Since $\cP$ is paracompact it suffices to show that $\cC_\cP^W$ is
  soft (see \cite{Bredon}*{Thm.\ 9.11}). To this end we claim that
  there is a smooth closed embedding $\pi:\Vfa/W\into\RR^n$ with
  $n=\dim\fa$. It suffices to prove this claim for $W$
  irreducible. Then either $W$ is finite in which case we can apply
  Chevalley's theorem or $W$ is the affine Weyl group attached to a
  finite root system $\Phi_0$ with Weyl group $W_0$ (observe that for
  every twisted affine root system there is an untwisted one with the
  same Weyl group). Then an embedding is provided by the smooth
  $W$-invariant functions
  $f_\omega(x):=\sum_{w\in W_0}\exp(2\pi i\omega(wx))$, where $\omega$
  runs through the fundamental weights of $\Phi_0$ (see, e.g.,
  \cite{Bou}*{VI, \S3.4, Thm.\ 1}).
  
  Since $\pi(\cP)$ is homeomorphic to $\cP$, it suffices to show that
  $\pi_*\cC_\cP$ is soft. Now observe that $\pi_*\cC_\cP$ is a
  $\cC_{\RR^ n}^\infty$-module because the $W_x$-invariance of a
  Taylor series is preserved by multiplication with a
  $W$-invariant. Thus, $\pi_*\cC_\cP$ is a module for the soft sheaf
  of algebras $\cC_{\RR^ n}^\infty$ and, therefore, itself soft (see
  \cite{Bredon}*{Thms.\ 9.16 and 9.17}).
\end{proof}

Now we investigate the cokernel of \eqref{eq:epsilon}. To this end,
consider the subgroup
\[
  A^\Phi:=\{u\in A\mid \tilde\alpha(u)=1\text{ for all
    $\alpha\in\Phi$}\}
\]
of \emph{$\Phi$-fixed points} of $A$.  By \eqref{e2} it is contained
in the subgroup $A^W$ of $W$-fixed points.

Of particular interest will be the group $A^{\Phi_x}$ and its
component group $\pi_0(A^{\Phi_x})$, where $x\in\cP$. If $y$ is close
to $x$ then $\Phi_y\subseteq\Phi_x$ and therefore
\[\label{eq:inclusion}
  A^{\Phi_x}\subseteq A^{\Phi_y}.
\]
This shows that there is a constructible sheaf $\fC_\cP$ on $\cP$ such
that $\pi_0(A^{\Phi_x})$ is its stalk at $x$ and the restriction maps
$\pi_0(A^{\Phi_x})\to\pi_0(A^{\Phi_y})$ are induced by
\eqref{eq:inclusion}. Its significance is given by

\begin{lemma}

  There is an exact sequence of sheaves of abelian groups
  \[\label{eq:CLC}
    \cC_\cP^W\overset{\epsilon^W}\to\fL_{\cP,\Lambda}^\Phi\overset\eta\to\fC_\cP\to0.
  \]

\end{lemma}

\begin{proof}

  Let $U\subseteq\cP$ be a convex, open neighborhood of $x\in\cP$ such
  that $\cP\cap U$ is convex, as well. If
  $\phi\in\fL^\Phi_{\cP,\Lambda}(U)$ then $\phi(x)\in A^{\Phi_x}$, by
  $\Phi$-equivariance. Thus we can define $\eta(\phi)(x)$ to be the
  image of $\phi(x)$ in $\pi_0(A^{\Phi_x})$.

  Now let $u\in A^{\Phi_x}$ be a representative of some element
  $\uq\in\pi_0(A^{\Phi_x})$. Then the constant map $\phi:x\mapsto u$
  is a section of $\fL^\Phi_{\cP,\Lambda}$ with $\eta(\phi)=\uq$, which
  shows that $\eta$ is surjective.

  On the other hand, for any section $f$ of $\cC_\cP^W$, the image
  $\epsilon^W(f)(x)$ lies in $\exp(\Vfa^{W_x})=(A^{\Phi_x})^0$, which
  shows $\|im|\epsilon^W\subseteq\ker\eta$.

  To show equality, let $\phi:U\to A$ be a section of
  $\fL^\Phi_{\cP,\Lambda}$ with $\eta(\phi)=0$, i.e.,
  $\phi(x)\in(A^{\Phi_x})^0$ for all $x\in U$. Then there is a lift
  $\tilde\phi:U\to\Vfa$ of $\phi$ with
  $\tilde\phi(x)\in\Vfa^{W_x}$. Since $\tilde\phi$ is smooth and
  closed there is a smooth function $f$ on $U$ with
  $\nabla f=\tilde\phi$ (\cref{lemma:nabla}). Let $\hat f$ be the
  Taylor series of $f$ in $x$. Then the $W_x$-equivariance of
  $\tilde\phi$ implies $\nabla({}^w\hat f)=\nabla\hat f$ for all
  $w\in W_x$. Hence ${}^w\hat f=\hat f+c_w$ with a constant
  $c_w\in\RR$. Evaluating this at $x$ implies $c_w=0$, i.e., $f$ is in
  fact $W_x$-invariant. Therefore, $\phi=\epsilon^W(f)$ is indeed in the
  image of $\epsilon^W$.
\end{proof}

To calculate the cohomology of $\fC_\cP$ we need a more explicit
description. The character group of $A^\Phi$ is given by
\[
  \Xi(A^\Phi)=\Lambda/\ZZ\VPhi.
\]
In particular, $\pi_0(A^\Phi)=0$ if and only if the root lattice
$\ZZ\VPhi$ is a direct summand of $\Lambda$. More generally, we have
\[
  \Xi(\pi_0(A^\Phi))=\|Tors|(\Lambda/\ZZ\VPhi)
  =\frac{\Lambda\cap\RR\VPhi}{\ZZ\VPhi}.
\]
Dualizing, this is equivalent to
\[\label{eq:pi0ApHi}
  \pi_0(A^\Phi)=\frac{(\ZZ\VPhi)^\vee}{\Lambda^\vee+(\RR\VPhi)^\vee},
\]
where $(\ZZ\VPhi)^\vee$ is the coweight lattice and $(\RR\VPhi)^\vee$
is the orthogonal complement of $\RR\VPhi$ in $\Vfa$.

We compute $\fC_\cP$ in two stages, the first treating the case of finite
root systems.

\begin{lemma}\label{lem:finiteC}

  Assume $\Phi$ is finite and $\Lambda=\ZZ\VPhi$. Then $\fC_\cP=0$.

\end{lemma}

\begin{proof}

  Let $S=\{\alpha_1,\ldots,\alpha_n\}$ be the set of simple roots of
  $\VPhi$. Since these form a basis of $\Vfa$ we get an isomorphism
  \[\label{eq:alpha*}
    \alpha_*:\Vfa\to\RR^n:x\mapsto(\<\alpha_1,x\>,\ldots,\<\alpha_n,x\>)
  \]
  For any subset $I\subseteq\{1,\ldots,n\}$ let $I'$ be its
  complement. For $k\in\{\RR,\ZZ\}$ we set
  \[
    k^I:=\{(x_i)\in k^n\mid x_i=0 \text{ for }i\in I'\}\cong k^{|I|}
  \]
  For a fixed $x\in\cP$, let $I:=\{i\mid\alpha_i(x)=0\}$. Then
  $\alpha_*$ maps $(\ZZ\VPhi_x)^\vee$, $(\RR\VPhi_x)^\perp$, and
  $\Lambda^\vee$ to $\ZZ^I\oplus\RR^{I'}$, $\RR^{I'}$, and $\ZZ^n$,
  respectively, and the claim follows from \eqref{eq:pi0ApHi}.
\end{proof}

Now assume that $\Phi$ is an infinite irreducible root system with
simple roots $S=\{\alpha_1,\ldots,\alpha_n\}$. The \emph{labels} of
$S$ are defined as the components of the unique primitive vector
$\delta:=(a_1,\ldots,a_n)\in\ZZ_{>0}^n$ such that
\[
  a_1\Valpha_1+\ldots+a_n\Valpha_n=0.
\]
For $I\subsetneq\{1,\ldots,n\}$ let again $I'\ne\leer$ be its
complement and
\[
  d_I:=\operatorname{gcd}\{a_j\mid j\in I'\}.
\]

\begin{lemma}\label{lem:infiniteC}

  Assume $\Phi$ is an infinite irreducible root system and
  $\Lambda=\ZZ\VPhi$. For any fixed $x\in\cA$ let $\fC_x$ be the stalk
  of $\fC_\cP$ in $x$ and $I=I_x:=\{i\mid\alpha_i(x)=0\}$. Then there is a
  canonical isomorphism
  \[
    \fC_x=\pi_0(A^{\Phi_x})\to\ZZ/d_I\ZZ.
  \]
  Moreover, this isomorphism is compatible with the restriction
  homomorphisms of $\fC_\cP$.

\end{lemma}

\begin{proof}

  We keep the notation of the proof of \cref{lem:finiteC}. The map
  $\alpha_*$ of \eqref{eq:alpha*} with $\alpha_i$ replaced by
  $\Valpha_i$ identifies $\Vfa$ with the hyperplane $H$ of $\RR^n$
  which is perpendicular to $\delta$. Thus \eqref{eq:pi0ApHi} becomes
  \[
    \pi_0(A^{\Phi_x})=\frac{(\ZZ^I\oplus\RR^{I'})\cap H}
    {(\RR^{I'}\cap H)+(\ZZ^n\cap H)}
  \]
  Now consider the homomorphism
  \[
    p_I:(\ZZ^I\oplus\RR^{I'})\cap
    H\to\ZZ/d_I\ZZ:(x_i)\mapsto\sum_{i\in I}a_ix_i+d_I\ZZ.
  \]
  Since $d_{I'}$ and $d_I$ are coprime there are $a',a\in\ZZ$ with
  $a'd_{I'}+ad_I=1$. Because $I'\ne\leer$, there is
  $(x_i)\in(\ZZ^I\oplus\RR^{I'})\cap H$ with
  $\sum_{i\in I}a_ix_i=a'd_{I'}$. Then $p_I(x_i)=1$, i.e., $p_I$ is
  onto.

  Next we claim that the kernel of $p_I$ is precisely
  $E:=(\RR^{I'}\cap H)+(\ZZ^n\cap H)$. Clearly
  $\RR^{I'}\cap H\subseteq\ker p_I$. Let $(x_i)\in\ZZ^n\cap H$. Then
  \[
    \sum_{i\in I}a_i x_i=-\sum_{j\in I'}a_jx_j\in d_I\ZZ
  \]
  shows that $E\subseteq\ker p_I$. To show the converse, let
  $(x_i)\in\ker p_I$. Then, by definition,
  $\sum_{i\in I}a_ix_i\in d_I\ZZ$. Hence there is $(y_i)\in\ZZ^{I'}$
  with
  \[
    \sum_{i\in I}a_i x_i=-\sum_{j\in I'}a_jy_j.
  \]
  Now define
  \[
    \overline x_i:=
    \begin{cases}
      x_i&\text{if }i\in I\\
      y_i&\text{if }i\in I'
    \end{cases}
  \]
  Then $(\overline x_i)\in\ZZ^n\cap H$ with
  $(x_i)-(\overline x_i)\in\RR^{I'}\cap H$, proving the claim. Thus
  $p_I$ induces an isomorphism between $\pi_0(A^{\Phi_x})$ and
  $\ZZ/d_I\ZZ$.

  For the final assertion, we denote $I$ by $I_x$. Let $y\in\cP$ be
  close enough to $x$ such that $I_y\subseteq I_x$. Then
  $d_{I_y}|d_{I_x}$. Thus, we have to show that the diagram
  \[
    \cxymatrix{ (\ZZ^{I_x}\oplus\RR^{I_x'})\cap
      H\ar[r]^>>>{p_{I_x}}\ar@{^(->}[d]&
      \ZZ/d_{I_x}\ZZ\ar@{>>}[d]^{[1]\mapsto[1]}\\
      (\ZZ^{I_y}\oplus\RR^{I_y'})\cap
      H\ar[r]^>>>{p_{I_y}}&\ZZ/d_{I_y}\ZZ}
  \]
  commutes. But this follows from $d_{I_y}|\,a_i$ for all
  $i\in I_x\setminus I_y$.
\end{proof}

From this we deduce:

\begin{lemma}

  Assume $\Lambda$ is of adjoint type. Then $\fC_\cP$ has a finite
  filtration such that each factor is a constant sheaf supported on a
  face of $\cP$.

\end{lemma}

\begin{proof}

  Let $\fa=\fa_0\times\fa_1\times\ldots\times\fa_m$ and
  $\Phi=\Phi_1\cup\ldots\cup\Phi_m$ be the unique decomposition of
  $(\fa,\Phi)$ into a trivial part $\fa_0$ and irreducible parts
  $\fa_1,\ldots,\fa_m$. The alcove $\cA$ of $\Phi$ will split
  accordingly as $\cA=\fa_0\times\cA_1\times\ldots\times\cA_m$. Then
  $\fC_P=\fC^{(1)}\oplus\ldots\oplus\fC^{(m)}$ where $\fC^{(i)}$ is
  the pull-back of $\fC_{\cA_i}$ to $\cP$. Thus it suffices to show
  the assertion for $\fC:=\fC^{(i)}$ for any $i$. By
  \cref{lem:finiteC} we may also assume that $\Phi_i$ is infinite. Let
  $\alpha_1,\ldots,\alpha_n\in\Phi_i$ be the simple roots.

  For any prime power $p^e$ let $\fC[p^e]\subseteq\fC_\cP$ be the
  kernel of multiplication by $p^e$. The union $\fC[p^\infty]$ over
  all $e$ is the $p$-primary component of $\fC_\cP$. Since $\fC_\cP$
  is the direct sum of its primary components it suffices to show the
  assertion for $\fC[p^\infty]$. Now it follows from
  \cref{lem:infiniteC} that $\fC[p^e]/\fC[p^{e-1}]$ is a constant
  sheaf with stalks $\ZZ/p\ZZ$ which is supported in the face
  \[
    \{x\in\cP\mid\alpha_i(x)=0\text{ for all $i$ with }p^e\nmid a_i\}
  \]
\end{proof}

Since constant sheaves on contractible spaces have trivial cohomology,
we get:

\begin{corollary}\label{cor:Cvanish}

  Assume $\Lambda$ is of adjoint type and that $\cP$ is convex. Then
  $H^i(\cP,\fC_\cP)=0$ for all $i\ge1$.

\end{corollary}

Next we study the kernel of $\epsilon^W$
(cf.~\eqref{eq:epsilon}). Observe that the constant sheaf $\RR_\cP$
is contained in the kernel. From this we get a homomorphism
\[\label{eq:barepsilon}
  \bar\epsilon^W:\cC_\cP^W/\RR_\cP\to\fL_{\cP,\Lambda}^\Phi
\]

\begin{lemma}\label{L4}

  Let $\fK_\cP$ be the kernel of $\bar\epsilon^W$. Then its stalk
  $\fK_x$ at $x\in\cP$ is equal to
  $\Lambda^\vee\cap(\RR\VPhi_x)^\vee$.

\end{lemma}

\begin{proof}

  Let $x\in\cP$ and let $U\subseteq\cP$ be a small convex open
  neighborhood. A smooth function $f$ on $U$ is in the kernel of
  $\epsilon$ if and only if its gradient is in
  $\Lambda^\vee$. Continuity implies that $\nabla f$ must be in fact
  constant. This implies that $f$ is an affine linear function with
  $\fq=\nabla f\in\Lambda^\vee$. Moreover, $f$ is a section of
  $\cC_\cP^W$ if and only if $\fq$ is $W_x$-invariant. This means
  $\fq$ should be orthogonal to all $\Valpha\in\VPhi_x$.
\end{proof}

In the following lemma let $\Lambda^\vee_\cP$ be the constant sheaf
with stalks $\Lambda^\vee$ on $\cP$. Similarly, $\ZZ_{H_i\cap\cP}$
will be the constant sheaf with stalks $\ZZ$ on $H_i\cap\cP$ which is
then extended by zero to $\cP$.

\begin{lemma}

  Let $\Lambda$ be of adjoint type and let $\alpha_1,\ldots,\alpha_n$
  be the simple roots of $\Phi$. Let $H_i$ be the hyperplane
  $\{\alpha_i=0\}$. Then the sheaf $\fK_\cP$ fits into an exact
  sequence
  \[\label{eq:KLZC}
    0\to\fK_\cP\to\Lambda^\vee_\cP\overset\rho\to
    \bigoplus_{i=1}^n\ZZ_{H_i\cap\cP}\overset\psi\to\fC_\cP\to0.
  \]
  If $\Phi$ is an infinite irreducible root system then $\psi$ maps
  the generator of the stalk $\ZZ_{H_i\cap\cP,x}$ in $x\in\cP$ to the
  class $a_i+d_I\ZZ$ (notation as in \cref{lem:infiniteC}).

\end{lemma}

\begin{proof}

  All sheaves are restrictions of the corresponding sheaves on $\cA$
  to $\cP$. Thus we may assume that $\cP=\cA$. Therefore we may treat
  every factor of the root system $\Phi$ separately. Thus, we may
  assume that $\Phi$ is either finite or irreducible and infinite.

  For $x\in\cP$ we have to show that the stalk
  $\fK_x=\Lambda^\vee\cap(\RR\VPhi_x)^\vee$ fits into an exact
  sequence
  \[\label{eq:kernelsequence}
    0\to\fK_x\to\Lambda^\vee\overset{\rho_x}\to
    \ZZ^{I_x}\overset{\psi_x}\to\fC_x\to0.
  \]
  First, we define $\rho_x$ as
  $\rho_x(v):=(\<\Valpha_i,v\>)_{i\in I_x}\in\ZZ^{I_x}$. Then $\fK_x$
  is the kernel of $\rho_x$ by \cref{L4}.

  If $\Phi$ is finite then the set of all $\Valpha_i$ with $i\in I_x$
  is part of a dual basis of $\Lambda^\vee$. Thus, $\rho_x$ is
  surjective and \eqref{eq:kernelsequence} is exact since $\fC_x=0$ in
  this case by \cref{lem:finiteC}.

  Now assume that $\Phi$ is irreducible and infinite. Then
  $\fC_x=\ZZ/d_{I_x}\ZZ$ by \cref{lem:infiniteC} and we define
  $\psi_x$ as $\psi_x(y_i):=\sum_{i\in
    I_x}a_iy_i+d_{I_x}\ZZ$. Identifying $\Vfa$ with the hyperplane $H$
  and $\Lambda^\vee$ with $\ZZ^n$ as in the proof of
  \cref{lem:infiniteC} we have to show that
  \[
    \ZZ^n\cap H\overset{\rho_x}\to
    \ZZ^{I_x}\overset{\psi_x}\to\ZZ/d_I\ZZ\to0
  \]
  is exact where $\rho_x$ is the projection
  $(x_i)\mapsto(x_i)_{i\in I_x}$. Surjectivity of $\psi_x$ follows
  again from $\operatorname{gcd}(d_I,d_{I'})=1$. Moreover, the kernel
  of $\psi_x$ consists of all $(y_i)_{i\in{I_x}}$ which can be
  extended to an $n$-tuple $(y_i)_{i=1}^n\in\ZZ^n$ with
  $\sum_{i=1}^na_iy_i=0$, since this is equivalent to
  $\sum_{i\in I_x}a_iy_i$ being divisible by $d_I$ (see proof of
  \cref{lem:infiniteC}).
\end{proof}

\begin{lemma}

  Assume $\Lambda$ is of adjoint type and that $\cP$ is convex. Then
  the homomorphism
  \[
    H^0(\psi):H^0(\cP,\bigoplus_i\ZZ_{H_i\cap\cP})\to H^0(\cP,\fC_\cP)
  \]
  is surjective.

\end{lemma}

\begin{proof}

  Both sides decompose as direct sums according to the decomposition
  of $\Phi$ into factors. Thus we may assume that $\Phi$ is
  irreducible. Then there is nothing to prove if $\Phi$ is finite
  since then $\fC_\cP=0$. So assume that $\Phi$ is infinite.

  Let $p$ be a prime and let $\fC[p]$ be the $p$-primary component of
  $\fC_\cP$. Then it suffices to show that the composition of
  $H^0(\psi)$ with the projection onto $H^0(\fC_\cP)[p]=H^0(\fC[p])$
  is surjective.

  To this end define the faces $\cF_j\subseteq\cA$, $j\ge0$, by the
  equations $\alpha_i=0$ with $p^j\nmid a_i$. Then
  $\cA=\cF_0\supseteq\cF_1\supseteq\ldots$ is a descending chain of
  faces of $\cA$. Let $e$ be maximal with
  $\cP_e:=\cF_e\cap\cP\ne\leer$. Then
  \[
    \cP=\cP_0\supseteq\cP_1\supset\ldots
    \supseteq\cP_e\supset\cP_{e+1}=\leer
  \]
  is a chain of of closed subsets of $\cP$. Let
  $\cP'_j:=\cP_j\setminus\cP_{j+1}=\cP\cap(\cF_j\setminus\cF_{j+1})$. Then
  the convexity of $\cP$ implies that also $\cP'_j$ is convex, hence
  contractible. Moreover, the explicit description of $\fC_\cP$ of
  \cref{lem:infiniteC} implies that the restriction $\fC_j' $ of
  $\fC[p]$ to $\cP'_j$ is locally constant with fiber $\ZZ/p^j\ZZ$. It
  follows that either $\cP'_j$ is empty or
  $H^0(\cP'_j,\fC[p])=\ZZ/p^j\ZZ$. Therefore, a global section of
  $\fC[p]$ is given by a system of elements of $H^0(\cP'_j,\fC[p])$
  which is compatible with the canonical restriction maps. This
  immediately implies $H^0(\cP,\fC[p])=H^0(\cP_e,\fC[p])=\ZZ/p^e\ZZ$.

  Finally, since the labels $a_*$ of $\Phi$ are coprime there is at
  least one label $a_i$ which is not divisible by $p$. Then
  $\cP_j\subseteq H_i\cap\cP$ for all $j\ge1$ and $\psi$ maps the
  generator of $\ZZ_{H_i\cap\cP}$ to $a_i+p^e\ZZ\in
  H^0(\cP,\fC[p])$ which yields the assertion.
\end{proof}

\begin{lemma}\label{lem:Kvanish}

  Assume $\Lambda$ is of adjoint type. Then $H^i(\cP,\fK_\cP)=0$ for
  all $i\ge2$.

\end{lemma}

\begin{proof}

  Let $\fT$ be the kernel of $\psi$, yielding a short exact sequence
  \[
    0\to\fT\to\bigoplus_i\ZZ_{H_i\cap\cP}\overset\psi\to\fC_\cP\to0.
  \]
  Since $H_i\cap\cP$ is convex, hence contractible, the higher
  cohomology of $\bigoplus_i\ZZ_{H_i\cap\cP}$ vanishes. We already
  proved that $H^i(\cP,\fC_\cP)=0$ for all $i\ge1$ in
  \cref{cor:Cvanish}. Combined with the surjectivity of $H^0(\psi)$
  this implies that $H^i(\cP,\fT)=0$ for all $i\ge1$. Since also
  $H^i(\cP,\Lambda^\vee)=0$ for all $i\ge1$, the short exact sequence
  \[
    0\to\fK_\cP\to\Lambda_\cP^\vee\to\fT\to0
  \]
  induced by \eqref{eq:KLZC} implies $H^i(\cP,\fK_\cP)=0$ for all
  $i\ge2$.
\end{proof}

\begin{proof}[Proof of Theorem \ref{T1}]

  By \cref{lem:commensurable} we may assume that $\Lambda$ is of
  adjoint type. Consider the short exact sequence
  \[
    0\to\RR_\cP\to\cC_\cP^W\to\cC_\cP^W/\RR_\cP\to0.
  \]

  The higher cohomology of the two sheaves on the left vanishes (see
  \cref{lemma:soft} for the second one) and thus so does that of the
  right hand sheaf. Let $\fS\subseteq\fL_{\cP,\Lambda}^\Phi$ be the
  image of $\epsilon^W$. Then we get a short exact sequence
  \[
    0\to\fK_\cP\to\cC_\cP^W/\RR_\cP\to\fS\to0
  \]
  from \eqref{eq:barepsilon}.
  \cref {lem:Kvanish} implies $H^i(\cP,\fS)=0$ for all
  $i\ge1$. Finally, \cref{cor:Cvanish} and the short exact sequence
  \[
    0\to\fS\to\fL^\Phi_{\cP,\Lambda}\to\fC_\cP\to0
  \]
  induced by \eqref{eq:CLC} imply $H^i(\cP,\fL^\Phi_{\cP,\Lambda})=0$
  for all $i\ge1$.
\end{proof}

\part{Classification of multiplicity free \qH\ manifolds}
\label{CMFQHM}

This part constitutes the main part of the paper. Except for
section \ref{sec:affine} on affine root systems, the content of the
preceding part will only be used in the proof of the
classification \cref{thm:main}.

\section{Quasi-Hamiltonian manifolds}\label{sec:loop2}

Recall the definition of Hamiltonian manifolds:

\begin{definition}

  Let $K$ be Lie group. A \emph{Hamiltonian $K$-manifold} is a
  $K$-manifold $M$ which is equipped with $2$-form $\omega$ and a
  smooth map $m:M\to\fk^*$ (the momentum map) such that

  \begin{enumerate}

  \item $m$ is $K$-equivariant,

  \item the $2$-form $\omega$ is $K$-invariant, closed and
    non-degenerate,

  \item $\omega(\xi x,\eta)=\<\xi,m_*\eta\>$ for all $\xi\in\fk$,
    $x\in M$, and $\eta\in T_xM$.

  \end{enumerate}

\end{definition}

The concept of \emph{\qH\ manifolds} is a multiplicative
version of Hamiltonian manifolds.  It was introduced in
\cite{AMM}*{\S8} (see also \cite{GS}*{\S1.4} for a short survey). The
main difference is that the momentum map has values in the Lie group
instead of the dual of the Lie algebra.

To define \qH\ manifolds one needs the Lie algebra $\fk$ to be equipped
with an $\|Ad|K$-invariant scalar product. This allows us to identify
$\fk^*$ with $\fk$.

Not essential but natural is to consider Lie groups with a
\emph{twist}, i.e., a fixed automorphism $k\mapsto{}^\tau k$ of
$K$. The target of the momentum map will then be $K$ as a set but with
$K$ acting on it by \emph{$\tau$-twisted conjugation}:
\[
  k*_\tau g:=k\cdot g\cdot{}^\tau k^{-1}.
\]
The set $K$ with this action will be denoted by $K\tau$. To
distinguish elements of the group $K$ from those of the space $K\tau$
we frequently denote the latter by $k\tau$ with $k\in K$.

Next, we need to introduce a couple of differential forms. First, let
$\theta$ and $\thetaq$ be the two canonical $\fk$-valued $1$-forms on
$K$ induced by left and right translation:
\[
  \theta(k\xi)=\xi=\thetaq(\xi k)\quad\text{with $\xi\in\fk$ and $k\in K$.}
\]
These are combined to a $\fk$-valued $1$-form on $K\tau$:
\[
  \Theta_\tau:=\half\big(\,\thetaq+{}^{\tau^{-1}}\theta\,\big).
\]
Explicitly
\[
  \Theta_\tau(k\xi)=\half(\Ad(k)\xi+{}^{\tau^{-1}}\xi).
\]
Using the scalar product on $\fk$ one defines the canonical
biinvariant closed $3$-form on $K$
\[
  \chi:={\textstyle\frac1{12}}\<\theta,[\theta,\theta]\>=
  {\textstyle\frac1{12}}\<\thetaq,[\thetaq,\thetaq]\>.
\]

\begin{definition}\label{D1}

  Let $K$ be a Lie group which is equipped with a $K$-invariant scalar
  product on $\fk$ and an automorphism $\tau$. A \emph{\qH\
    $K\tau$-manifold} is a $K$-manifold $M$ equipped with a $2$-form
  $\omega$ and a smooth map $m:M\to K\tau$ (the {\it group valued
    momentum map}) such that:

  \begin{enumerate}

  \item\label{D1i0} The map $m:M\to K\tau$ is $K$-equivariant,

  \item\label{D1i1} the form $\omega$ is $K$-invariant and satisfies
    $\mathrm{d}\omega=-m^*\chi$,

  \item\label{D1i3} $\omega(\xi x,\eta)= \<\xi,m^*\Theta_\tau(\eta)\>$
    for all $\xi\in\fk$, $x\in M$, and $\eta\in T_xM$,

  \item\label{D1i2}
    $\ker\omega_x=\{\xi x\in T_xM\mid\xi\in\fk\text{ with } \Ad
    m(x)({}^{\tau}\xi)+\xi=0\}$.

  \end{enumerate}

\end{definition}

\begin{remark}

  This definition originates very naturally from studying (ordinary)
  Hamiltonian actions of the twisted loop group
  \[
    \cL_\tau K:=\{\phi:\RR\to K\mid\phi(t+1)={}^\tau\phi(t)\}.
  \]
  See \cite{AMM} for the untwisted case and \cite{Meinrenken} or the
  first version of this paper \cite{KnopQHam0} for the twisted version.
\end{remark}

For the proof of \cref{prop:equivMF} we will need the following
observation:

\begin{lemma}\label{lemma:mf1}
  Let $M$ be a \qH\ manifold, $x\in M$, and
  $E_x:=\ker(1+\|Ad|m(x)\circ\tau)\subseteq\fk$. Then
  \[\label{eq:mf1}
    (\fk x)^\perp=\ker D_xm+E_xx.
  \]
\end{lemma}

\begin{proof}
  The inclusion \8$\supseteq$\9 follows from part \ref{D1i3} of
  \cref{D1} and $E_xx=\ker\omega_x$ (part \ref{D1i2}). For the
  opposite inclusion let $\eta\in(\fk x)^\perp$. Put $a:=m(x)$ and
  $\sigma:=\half m_*(\eta)a^{-1}\in\fk$. Then part \ref{D1i3} implies
  $\sigma\in E_x$. Moreover, the equivariance of $m$ (part \ref{D1i0})
  implies
  \[
    m_*(\sigma x)=\sigma*_\tau a=\sigma a-a\otau\sigma=2\sigma
    a=m_*(\eta).
  \]
  Thus if $\rho:=\eta-\sigma x$ then $\eta=\rho+\sigma x$ with
  $\rho\in\ker D_xm$ and $\sigma x\in E_xx$.
\end{proof}

\section{Twisted conjugacy classes}\label{sec:twisted}

In this section we recall the geometry of $K\tau$ as a
$K$-manifold. As sources we use mostly the papers
% \cite{Segal}, \cite{BroeckerTomDieck}*{Ch.\ IV.4},
\cite{Wendt}, \cite{MW}, and \cite{Meinrenken}. From now on $K$ is
assumed to be a \emph{simply connected} and \emph{compact} Lie
group. For our take on affine root systems see section
\ref{sec:affine}.

Let $T\subseteq K$ be a maximal torus, $\ft=\|Lie|T$ its Lie algebra,
and $\Xi(T)$ its character group. For any $\chi\in\Xi(T)$ there is a
unique $a_\chi\in\ft$ with
\[
  \chi(\exp\xi)=e^{2\pi i\<a_\chi,\xi\>}\quad\text{for all $\xi\in\ft$}.
\]
Then $\chi\mapsto a_\chi$ identifies $\Xi(T)$ with a lattice in
$\ft$. In particular, the root system $\VPhi(\fk,\ft)\subseteq\Xi(T)$
of $\fk$ can be considered as a subset of $\ft$.

\begin{theorem}\label{thm:conjugacyclasses}

  Let $K$ be a simply connected compact Lie group $K$ and $\tau$ an
  automorphism of $K$. Then there is a $\tau$-stable maximal torus
  $T\subseteq K$ and an integral affine root system
  $(\Phi_\tau,\Lambda_\tau)$ on $\fa=\ft^\tau$, the $\tau$-fixed part
  of $\ft=\Lie T$, with the following properties:

  \begin{enumerate}

  \item\label{it:Comparison} (Comparison) Let $\pr_\fa:\ft\to\fa$ be
    the orthogonal projection. Then $\VPhi_\tau=\pr_\fa\VPhi(\fk,\ft)$
    and $\Lambda_\tau=\pr_\fa\Xi(T)$. Moreover,
    $\Lambda_\tau=\<\VPhi_\tau^\vee\>^\vee$ (the \emph{weight lattice of
    $\Phi_\tau$}).

  \item\label{it:OrbitSpace} (Orbit space) Let $\cA\subseteq\fa$ be an
    alcove of $\Phi_\tau$.  Then the composition
    \[\label{eq:AK}
      c:\cA\subseteq\fa\overset{\exp}\Pfeil K\to K\tau/K
    \]
    is a homeomorphism.

  \item\label{it:Orbits} (Orbits) For $a\in\cA$ let $u=\exp(a)$,
    \[\label{eq:Ka}
      K_{a\tau}:=\{k\in K\mid k*_\tau u=u\},
    \]
    and
    \[
      \Phi_\tau(a):=\{\alpha\in\Phi_\tau\mid\alpha(a)=0\}.
    \]
    Then $K_{a\tau}$ is a connected subgroup of $K$ with maximal torus
    $\exp\fa$ and root datum $(\overline{\Phi_\tau(a)},\Lambda_\tau)$.

  \end{enumerate}

  % The quadruple $(T,\fa,\Phi_\tau,\Lambda_\tau)$ is uniquely
  % determined by $K$ and $\tau$ up to conjugation by $K_\tau=K^\tau$.
\end{theorem}

\begin{proof}
  Morally, all assertions are contained in the papers \cite{Wendt},
  \cite{MW}, and \cite{Meinrenken} but it is a bit unclear to which
  degree of generality they are proven. Unproblematic is the case when
  $K$ is simple and $\tau$ is a diagram automorphism, i.e., is induced
  by an automorphism of the Dynkin diagram: Parts \ref{it:Comparison}
  and \ref{it:OrbitSpace} follow from the discussion in
  \cite{MW}*{\S3} while \ref{it:Orbits} is \cite{MW}*{\S4}.
  
  The general case follows easily.

  \emph{Step 1: Reduction to \8$K$ simple\9.} If $K$ decomposes into a
  nontrivial product $K_1\times K_2$ with $\tau$-invariant factors
  then all assertions reduce to the factors $K_i$. So assume that this
  is not the case. Then $\tau$ permutes the simple factors of $K$
  cyclically. Thus we may assume that $K=K_0\times\ldots\times K_0$
  (with $m$ simple factors) and that $\tau$ acts as
  \[
    \otau(k_1,k_2,\ldots,k_m)=(k_2,\ldots,k_m,{}^{\tau_0}k_1)
  \]
  where $\tau_0\in\|Aut|K_0$.  The $\tau$-twisted action then has the
  form
  \[
    (k_i)\ (g_i)\,\otau(k_i)^{-1}=(k_1g_1k_2^{-1},\ldots,
    k_{m-1}g_{m-1}k_m^{-1},k_mg_m{}^{\tau_0}k_1^{-1}).
  \]
  It follows that the twisted action of $1\times K_0^{m-1}\subseteq K$
  is free with quotient map
  \[
    K\tau\to K_0\tau_0:(g_1,\ldots,g_m)\mapsto g_1\ldots g_m
  \]
  which is equivariant with respect to the first copy of $K_0$. This
  implies easily that $T=T_0^m$, $\fa=\fa_0:=\ft_0^{\tau_0}$
  diagonally embedded in $\fa_0^m\subseteq\ft$,
  $\Lambda_\tau=\Lambda_{\tau_0}$, and $\Phi_\tau=m^*\Phi_{\tau_0}$
  have all required properties, where
  $(m^*\alpha)(a):=\frac1m\alpha(ma)$.

  \emph{Step 2: Reduction to \8$\tau$ diagram automorphism\9.} Fix a
  maximal torus $T_0\subseteq K$ and consider all diagram automorphism
  with respect to $T_0$ (and a choice of positive roots). Since these
  represent all classes of $\|Out|K$ there is $h\in K$ and a diagram
  automorphism $\tau_0$ with $\tau=\Ad(h)\circ\tau_0$. Let $\cA_0$ be
  an alcove for $\Phi_{\tau_0}$. Then, by \ref{it:OrbitSpace}, there
  are $a\in\cA_0$ and $t\in K$ with $h=t*_{\tau_0}u$ where
  $u:=\exp a$. Then a short calculation yields
  \[
    \tau=\Ad(t)\,\Ad(u)\tau_0\,\Ad(t)^{-1}.
  \]
  The automorphism $\tau_1:=\Ad(u)\tau_0=\tau_0\Ad(u)$ is intertwined
  with $\tau_0$ via right translation by $u$. So, $\Phi_{\tau_1}$ is
  the root system $\Phi_{\tau_0}$ translated by $-a$ living on the
  same torus $T_1=T_0$. Finally, put $T=\Ad(t)T_0$ and define
  $\Phi_\tau$ on $T$ by \8transport of structure\9.
\end{proof}

\begin{remark}
  The quadruple $(T,\fa,\Phi_\tau,\Lambda_\tau)$ is actually uniquely
  determined by $K$ and $\tau$ up to conjugation by
  $K_\tau=K^\tau$. To see this, observe that $\fa$ is, by
  \ref{it:Orbits}, a Cartan subspace of $\fk^\tau$, hence unique up to
  conjugation by $K^\tau$.  Then there is only one choice for $T$
  namely the centralizer of $\fa$ in $K$. For this one has to prove
  that $\fa$ is regular, i.e., no root of $K$ vanishes on $\fa$. Since
  $a\tau$ with $a\in\exp\fa$ and $\tau$ yield the same subspace $\fa$
  one may assume that $\tau$ is a diagram automorphism. But then $\fa$
  is defined by the equalities $\alpha={}^\tau\alpha$ where $\alpha$
  runs through all simple roots of $K$. Thus $\fa$ contains the sum of
  the fundamental coweights which clearly has a nonzero scalar product
  with every root. Finally assertion \ref{it:OrbitSpace} determines
  an alcove $\cA$. Then $\Phi_\tau$ and $\Lambda_\tau$ are again
  determined by property \ref{it:Orbits}.
\end{remark}

Next we describe the local structure of $K\tau$. For $a\in\cA$ let
$\cC_a:=\RR_{\ge0}(\cA-a)$ be the tangent cone of $\cA$ in $a$. It is
a Weyl chamber of $\Lie K_{a\tau}$.

\begin{corollary}\label{cor:LocStr}
  Keep the notation from \cref{thm:conjugacyclasses} and let
  $a\in\cA$.  There exists an open neighborhood $U$ of $a$ in $\cA$
  such that $U-a$ is an open neighborhood of $0$ in $\cC_a$ and such
  that
  \[\label{eq:LokStrK}
    K\times^L\fl_{U-a}\overset\sim\longrightarrow(K\tau)_U:
    [k,\xi]\mapsto k*_\tau(\exp(\xi)\exp(a))
  \]
  is a $K$-diffeomorphism where $L=K_{a\tau}$,
  $\fl_{U-a}:=\fl\times_{\cC_a} (U-a)$, and
  $(K\tau)_U:=K\tau\times_\cA U$.
\end{corollary}

\begin{proof}
  Consider the twisted orbit $Y:=K*_\tau u$ of $u:=\exp(a)$ in
  $K\tau$. The isotropy group of $u$ is $L$. Because of $\tau(u)=u$ we
  have $u*_\tau u=u$ and therefore $u\in L$. Then it has been shown in
  \cite{Meinrenken}*{Prop.\ 2.5} that $L\tau\subseteq K\tau$ is a
  slice of $Y$ in $u$. This means that $L\tau$ is $L$-stable and that
  it is transversal to $Y$ in $u$. The slice theorem implies that
  $K\times^LL\tau\to K\tau:[k,l\tau]\mapsto (k*_\tau l)\tau$ is a
  diffeomorphism on a $K$-invariant neighborhood of $K\times^L\{u\}$.

  Now observe that $l\in L$ means $lu\otau l^{-1}=l*_\tau u=u$ and
  therefore ${}^\tau l=u^{-1}lu$. Hence the $\tau$-twisted action of
  $L$ on itself is conjugation twisted by $\Ad u^{-1}$. Thus, the map
  $L\tau\to L:l\tau\mapsto lu^{-1}$ intertwines the twisted
  conjugation on $L$ with the usual conjugation action. Moreover, it
  sends $u$ to $e\in L$. This implies that the map
  $K\times^LL\to K\tau:[k,l]\mapsto k*_\tau(lu)$ is a diffeomorphism
  near $K\times^L\{e\}$ where $L$ acts on itself by conjugation. The
  assertion follows now from the fact that $\exp:\fl\to L$ is a
  diffeomorphism in a small open neighborhood $\fl_{U-a}$ of $0$.
\end{proof}

\begin{Remarks}
  \textit{a)} The type of the root system $\Phi_\tau$ depends only on
  the image $\overline\tau$ of $\tau$ in $\|Out|(K)$. If $K$ is simple
  let $\sX_n$ with $\sX\in\{\sA,\sB,\sC,\sD,\sE,\sF,\sG\}$ be the type
  of the Dynkin diagram of $K$. Then $\Phi_\tau$ is the affine root
  system of type $X_n^{(r)}$ in the notation of \cite{Kac} where
  $r\in\{1,2,3\}$ is the order of $\overline\tau$. The general case is
  reduced to this one by the first reduction in the proof of
  \cref{thm:conjugacyclasses}.

  \textit{b)} As seen in the proof of \cref{thm:conjugacyclasses},
  multiplying $\tau$ with $\|Ad|u$, where $u=\exp(a)$ with $a\in\cA$
  results in a translation of the root system $\Phi_\tau$ by
  $-a$. Accordingly the alcove $\cA$ will also be translated which
  means that $0$ may no longer be a vertex. Take, e.g.,
  $K=\SU(2)$ and let $\otau k=\overline k$ (complex conjugation). Then
  $\tau$ is also conjugation by
  $j:=\begin{pmatrix}0&-1\\1&0\end{pmatrix}$. Thus we can take
  $T=\SO(2)$ and have $\fa=\ft=\RR j$. Since $\cA_{\|id|_K}=[0,\pi]j$
  and $j=\exp(\frac\pi2 j)$ we have
  $\cA_\tau=\cA_{\|id|_K}-\frac\pi2j=[-\frac\pi2,\frac\pi2]j$. Thus
  $\exp\cA_\tau$ is a fundamental domain for the twisted action
  $k*_\tau g=kg\overline k^{-1}=kgk^t$.

\end{Remarks}

\section{The local structure of quasi-Hamiltonian manifolds}
\label{sec:localstructure}

Now we describe how to transfer the results on the local structure of
$K\tau$ to \qH\ manifolds. Thereby we follow mostly \cite{AMM} and
\cite{Meinrenken}.

Let $m:M\to K\tau$ be the momentum map of a \qH\ manifold. Then the
inverse of the homeomorphism $c:\cA\to K\tau/K$ of \eqref{eq:AK}
yields a map $m_+=c^{-1}\circ \pi\circ m$, the \emph{invariant
  momentum map}, which fits into the commutative diagram
\[
  \dxymatrix{
    M\ar[r]^m\ar[d]^{m_+}&K\tau\ar[d]^\pi\\\cA\ar[r]^>>>>c&K\tau/K }
\]
The image $m_+(M)\subseteq\cA$ will be called the \emph{momentum
  image} of $M$ and will be denoted by $\cP_M$.  Our first goal is to
describe $M$ locally over $\cP_M$.

Let $a$, $L$ and $U$ be as in \cref{cor:LocStr} and assume
$m_+(M)\subseteq U$. Put
\[
  \log_LM:=M\times_{K\tau}\fl=M\times_{(K\tau)_U}\fl_{U-a}
\]
(with respect to the map $\fl\to
K\tau:\xi\mapsto\exp(\xi)\exp(a)$). Then \eqref{eq:LokStrK} shows that
\[
  K\times^L\log_LM\to M:[k,(m,\xi)]\mapsto km
\]
is a homeomorphism. Using the two projections $\iota:\log_LM\into M$
and $m_0:\log_LM\to\fl\cong\fl^*$ one defines the $2$-form
$\omega_0=\iota^*\omega-m_0^*\tilde\omega$ on $\log_LM$ where
$\tilde\omega$ is the $2$-form on $\fl$ which was defined in
\cite{AMM}*{Lemma~3.3}.

\begin{remark}
  Let $(M_0,m_0,\omega_0)$ be a Hamiltonian $K$-manifold. Then
  completely analogously to $m_+$ one can define the invariant
  momentum map $(m_0)_+$ for $M_0$. It has values in a Weyl chamber
  $\ft^+$ inside the dual Cartan subalgebra $\ft^*$ of $\fk$. For
  details see \cite{KnopAutoHam}*{Eqn.~(2.1)}.
\end{remark}

\begin{theorem}\label{thm:localstructure}
  Let $a$, $L$ and $U$ be as in \cref{cor:LocStr}.

  \begin{enumerate}

  \item Let $(M,m,\omega)$ be a \qH\ $K\tau$-manifold with
    $m_+(M)\subset U$. Then the triple $(\log_LM,m_0,\omega_0)$ is a
    Hamiltonian $L$-manifold.

  \item $\log_L$ is an equivalence of the category of \qH\
    $K\tau$-manifolds $M$ with $m_+(M)\subseteq U$ and the category of
    Hamiltonian $L$-manifolds $M_0$ with $(m_0)_+(M_0)\subseteq U-a$
    where in both categories the morphisms are the
    isomorphisms. Moreover, as a manifold, we have $M=K\times^L\log_LM$.

  \item Let $m_+(M)\subseteq U$. Then $\log_L$ preserves the momentum
    image in the sense that $m_+(M)=(m_0)_+(\log_LM)+a$.

  \end{enumerate}

\end{theorem}

\begin{proof}
  \emph{a)} The construction of $\log_LM$ can be performed in three
  steps. For the first observe that $L=K_{a\tau}$ is
  $\tau$-stable. Hence the inclusion $L\tau\into K\tau$ is
  $L$-equivariant. By \cite{Meinrenken}*{Prop.~4.1} the preimage
  $M_2:=M\times_{K\tau}L\tau$ has the structure of a \qH\
  $L\tau$-manifold.  For the second step put $M_1:=M_2$ with momentum
  map changed to $x\mapsto m(x)u^{-1}$. Recall from the proof of
  \cref{cor:LocStr} that the map $L\tau\to L:l\tau\mapsto lu^{-1}$
  intertwines the twisted conjugation on $L$ with the usual
  conjugation. Then $M_1$ is an untwisted
  \qH\ $L$-manifold. Finally, $\log_LM=M_0$ is the pull-back of $M_1$
  via the exponential map $\exp:\fl\to L$. Now the assertion follows
  from \cite{AMM}*{Remark 3.3}.

  \emph{b)} For the inverse functor we invert each of the three steps
  above separately. We start with a Hamiltonian $L$-manifold
  $(M_0,m_0,\omega_0)$ satisfying $(m_0)_+(M_0)\subseteq U-a$. Then
  according to \cite{AMM}*{Prop.\ 3.4} the triple
  $(M_1,m_1,\omega_1):=(M_0,\exp\circ q\circ
  m_0,\omega_0+m_0^*\tilde\omega)$ is an (untwisted) \qH\ $L$-manifold
  with $(m_1)_+(M_1)\subseteq U-a$. Here $q$ is the identification
  $\fl^*\overset\sim\to\fl$. Observe that Prop.~3.4 of \cite{AMM} is
  applicable since the exponential function is locally invertible on
  $\fl_{U-a}$. Thus, the functor $M_0\mapsto M_1$ inverts the functor
  $M_1\mapsto M_0$.

  In the second step we put
  $(M_2,m_2,\omega_2):=(M_1,m_1\cdot u,\omega_1)$. Then $M_2$ is an
  $\Ad(u^{-1})$-twisted \qH\ $L$-manifold with
  $(m_2)_+(M_2)\subseteq U$ satisfying $m_+(M)\subseteq U$. This
  inverts the functor $M_2\mapsto M_1$.

  It remains to invert the functor $M\mapsto M_2=m^{-1}(L)$. This means
  in particular that we need to provide $M:=K\times^LM_2$ with a \qH\
  $K\tau$-structure whenever $M_2$ is an $\Ad u^{-1}$-twisted \qH\
  $L$-manifold with $(m_2)_+(M_2)\subseteq U$.

  To this end recall the double $D(K)$ of a group $K$ from
  \cite{AMM}*{\S3.2}. As a set it is $D(K)=K\times K$ with
  $K\times K$-action $(u,v)*(a,b)=(uav^{-1},vbu^{-1})$ and
  $K\times K$-valued momentum map $m(a,b)=(ab,a^{-1}b^{-1})$. Now we
  consider $K\tau$ as a connected component of the semidirect product
  $\tilde K:=K\rtimes\ZZ\tau$. Then inside $\tilde K$ products like
  $b\cdot\tau$ and
  $a(b\cdot\tau) a^{-1}=ab(\tau a^{-1}\tau^{-1})\tau=(a*_\tau
  b)\cdot\tau$ make sense. The latter formula shows that
  $\tau$-twisted conjugation becomes ordinary conjugation on $K\tau$.

  Now we define the \emph{twisted double} $D_\tau(K)$ as the connected
  component $K\times K\tau$ inside
  $D(\tilde K)=\tilde K\times\tilde K$. Then $D_\tau(K)$ is
  $K\tau\times K\tau^{-1}$-\qH\ with $K\times K$-action
  \[
    (u,v)*(a,b\cdot\tau)=(uav^{-1},vb\tau u^{-1})= (uav^{-1},vb\,\otau
    u^{-1}\cdot\tau)
  \]
  and momentum map
  \[
    m_D(a,b\cdot\tau)=(ab\tau,a^{-1}\tau^{-1}b^{-1})=
    (ab\tau,a^{-1}\,{}^{\tau^{-1}}b^{-1}\cdot\tau^{-1}).
  \]
  Let
    \[
      Z:=m_D^{-1}(K\tau\times(L\tau)_U^{-1})=\{(b,c\tau)\in
      D_\tau(K)\mid c\otau b\in (L\tau)_U\}.
  \]
  where $(L\tau)_U$ is the open subset
  $L\tau\cap(K\tau)_U=L\tau\times_\cA U$ of $L\tau$ and
  $(L\tau)_U^{-1}\subseteq K\tau^{-1}$ is the set of its
  inverses. Again by \cite{Meinrenken}*{Prop.~4.1} this is a \qH\
  $K\tau\times L\tau^{-1}$-manifold. One easily checks that the map
  \[\label{eq:Zident}
    K\times (L\tau)_U\to Z:(a,l\tau)\mapsto(a,l\tau a^{-1})
  \]
  is invertible, yielding a \qH\ structure on $K\times (L\tau)_U$ with
  $K\tau\times L\tau^{-1}$-valued momentum map.

\renewcommand{\thefootnote}{\fnsymbol{footnote}}
  
Now consider the fusion product $Z\otimes_LM_2$ with respect to the
factor $L$ (see \cite{AMM}*{\S6}\footnote{Strictly speaking, the paper
  \cite{AMM} deals only with the untwisted case but as explained above
  one may consider twisted $K\tau$-manifolds as open subsets of
  (untwisted) $\tilde K$-manifolds to which \cite{AMM} applies.}.)
which equals $Z\times M_2$ as a manifold and is $K\tau\times L$-\qH\
with momentum map
  \[
    m_\otimes:(a,l\tau,x)\mapsto(al\tau a^{-1},(l\tau)^{-1}m(x))\in K\tau\times
    L
  \]
  where we have identified $Z$ with $K\times(L\tau)_U$ using
  \eqref{eq:Zident}. The $q$-Hamiltonian reduction
  \[\label{eq:expL}
    M':=(Z\otimes_LM_2)/\!\!/L:=m_\otimes^{-1}(K\tau\times\{1\})/L
  \]
  defined in \cite{AMM}*{\S5}$^{\scriptstyle*}$ with respect to $L$ is
  a \qH\ $K\tau$-manifold. It is diffeomorphic to $M:=K\times^LM_2$
  via the map
  \[
    M\to M':[a,x]\mapsto(a,m(x),x).
  \]
  This provides $M$ with a $K\tau$-\qH\ structure. Since all
  constructions are functorial for isomorphisms we have defined a
  functor $M_2\mapsto M$. It remains to show that it is quasiinverse
  to the functor $M\mapsto M_2$ defined above.

  We start with the composition $M_2\mapsto M\mapsto M_2$ which needs
  to be isomorphic to the identity functor. For this, let $M_2$ be as
  above. Then it suffices to show that the map
  \[
    \phi:M_2\to (K\times^LM_2)\times_{K\tau}L\tau:x\mapsto([1,x],m_2(x))
  \]
  is an isomorphism of \qH\ $L\tau$-manifolds. Tracing through all
  definitions one sees easily that $\phi$ is a diffeomorphism which is
  compatible with the $L$-actions and the momentum maps. It is more
  difficult to see that that the $2$-forms match up, as well.

  For this observe that $x\in M_2$ is mapped by $\phi$ to the $L$-orbit of
  $(1,m_2(x),x)\in K\times L\times M_2\subseteq D_\tau(K)\times M_2$.
  Recall the explicit formula of \cite{AMM}*{Thm.~6.1} for the
  $2$-form on the fusion product of two \qH\ manifolds
  $(M',\omega',m')$ and $(M'',\omega'',m'')$:
  \[\label{eq:fusion}
    \pi_1^*\omega'+\pi_2^*\omega''+\half\<{m'}^*\theta,{m''}^*\thetaq\>.
  \]
  We apply this formula to $M'=Z$ and $M''=M_2$.  Let
  $\iota:L\to L:h\mapsto h^{-1}$ be the inversion. Because of
  $\iota^*\theta=-\thetaq$ the pull-back of the third term in
  \eqref{eq:fusion} to $M_2$ vanishes. To see that also the first
  summand vanishes on $M_2$ we look at the explicit form of $\omega_D$
  on $D(\tilde K)$ (see \cite{AMM}*{Prop.~3.2}):
  \[
    \omega_D=\half\<p_1^*\theta,p_2^*\thetaq\>+\half\<p_1^*\thetaq,p_2^*\theta\>
  \]
  where $p_1,p_2$ are the two projections of $D(\tilde K)$ to
  $\tilde K$. The map from $M_2$ to $Z\subseteq D_\tau(K)$ is
  $x\mapsto(1,m_2(x))$. Hence $p_1$ is constant on $M_2$ implying that
  the pull-backs of $p_1^*\theta$ and $p_1^*\thetaq$, hence of
  $\omega_D$ to $M_2$ vanish. This finishes the proof that $\phi$ is an
  isomorphism of quasi-Hamiltonian manifolds.

  Now let $(M,m,\omega)$ be given with $m_*(M)\subseteq U$ and
  $M_2:=M\times_{K\tau}L\tau$ as in part \emph{a)}. Then we show that
  \[
    \psi:K\times^LM_2\to M:[k,x]\mapsto kx
  \]
  is an isomorphism of quasi-Hamiltonian $K\tau$-manifolds. As above it is
  easy to check that $\psi$ is a $K$-equivariant diffeomorphism which
  is compatible with the momentum maps. At stake are again the
  $2$-forms.
  
  The map $\psi$ defines a second $2$-form $\omega'$ on $M$ such that
  $(M,m,\omega')$ is \qH. Moreover, by the previous discussion the
  pull-backs of $\omega$ and $\omega'$ to $M_2$ coincide. Thus, to
  show $\omega=\omega'$ it suffices to show that $\omega$ is uniquely
  determined by $m$ and its pullback to $M_2$. Let $x\in M$. By
  $K$-invariance we may assume that $x\in M_2$. The momentum map
  property \ref{D1i3} of \cref{D1} allows to compute $\omega_x(\xi,\eta)$
  where $\xi\in\fk x$ and $\eta\in T_xM$. Moreover, also
  $\omega_x(\xi,\eta)$ is known for $\xi,\eta\in T_xM_2$. Because of
  $\fk x+T_xM_2=T_xM$ we proved our assertion.
  
  \emph{c)} is obvious.
\end{proof}

\cref{thm:localstructure} can be used to analyze the local structure
of an arbitrary $M$. To enforce the requirement $m_+(M)\subseteq U$ we
apply the theorem to the open part
\[
  M_U:=m_+^{-1}(U)=M\times_\cA U\subseteq M.
\]
There is a subtlety however in that $M_U$ might not be connected
despite $M$ being so. This is a problem since some fundamental
structure theorems like convexity of the momentum image $m_+(M)$ or
connectedness of fibers of $m$ depend crucially on connectedness of
$M$. Therefore we impose an extra condition on $M$ which is automatic
in case $M$ is compact.

\begin{definition}
  A (quasi-)Hamiltonian manifold $(M,m)$ is \emph{locally convex} if
  all fibers of $m_+$ are connected and $m_+:M\to m_+(M)$ is an open
  map. It is \emph{convex} if in addition $\cP_M=m_+(M)$ is convex.
\end{definition}

\begin{lemma}\label{lemma:conv}
  Let $M$ be a (quasi-)Hamiltonian manifold.

  \begin{enumerate}

  \item\label{it:conv1} If $M$ is locally convex then $M_U$ is locally
    convex for all $U\subseteq\cA$ open. If additionally,
    $\cP_M\cap U$ is connected then $M_U$ is connected.
 
  \item If $M$ is locally convex then the momentum image $\cP_M$ is
    locally polyhedral (hence locally convex, locally closed, and
    locally connected). In particular, every $a\in\cP_M$ has a
    neighborhood $U$ such that $\cP_M\cap U$, and therefore $M_U$, is
    convex.

  \item Assume that all fibers of $m_+$ are connected and that $m_+$
    is proper onto $\cP_M$. Then $M$ is locally convex.
    
  \item\label{it:conv4} Assume $M$ is compact. Then $M$ is convex and
    $\cP_M$ is a polytope.

  \end{enumerate}

\end{lemma}

\begin{proof}
  \emph{a)} is easy point set topology.

  \emph{b)} Using \cref{thm:localstructure} and \emph{a)} we may
  assume that $M$ is Hamiltonian. Then \cite{KnopConvexity}*{Cor.~2.8,
    Thm.~5.1} implies that every $x\in M$ has a $K$-invariant open
  neighborhood $V\subseteq M$ such that $m_+(V)$ is open in $C_x$
  where $C_x$ is a polyhedral cone with vertex $m_+(x)$. Since $m_+$
  is open, $m_+(V)$ is also open in $m_+(M)$.

  \emph{c)} It suffices to show that every $x\in M$ has an open
  neighborhood $W\subseteq M$ such that $W\to m_+(M)$ is an open
  map. Again by \cref{thm:localstructure} we may assume that $M$ is
  Hamiltonian (possibly not connected). Choose a neighborhood $V_x$ of
  $x$ and let $a=m_+(x)$ and $C_x$ be as above. By
  \cite{KnopConvexity}*{Cor.~6.1}, the cone $C_x$ is locally constant
  in $x\in m_+^{-1}(a)$. Since the fiber $F=m_+^ {-1}(a)$ is connected
  this implies that $C=C_x$ is the same for all $x\in F$. The fiber
  $F$ being compact, there are finitely many $x_1,\ldots,x_n\in F$
  with $F\subseteq V:=V_{x_1}\cup\ldots\cup V_{x_n}$. Because all
  restrictions $V_{x_i}\to C$ of $m_+$ are open, the same holds for
  $V$. Because of properness, the set $m_+(M\setminus V)$ is closed in
  $m_+(M)$. Hence the complement
  \[
    U=m_+(M)\setminus m_+(M\setminus V)=\{b\in m_+(M)\mid
    m^{-1}(b)\subseteq V\}
  \]
  is an open neighborhood of $a$ in $m_+(M)$ with $U\subseteq
  m_+(V)$. Thus, $W:=m_+^{-1}(U)\subseteq V$ has the required
  property.

  \emph{d)} This is \cite{Meinrenken}*{Thm.~4.4}. Openness of $m_+$
  follows from \emph{c)}.
\end{proof}

Besides the momentum image $\cP_M$, the most important invariant of a
(quasi)-Hamiltonian is its principal isotropy group. We briefly recall
its definition.  Before that we need to make the following

\begin{remark}
  Assume $M$ is \qH\ with $m_+(M)\subseteq\cF$ where $\cF$ is a face
  of $\cA$. Let $a\in m_+(M)$ and let $U\subseteq\cA$ be as in
  \cref{thm:localstructure}. Then $m^{-1}((L\tau)_U)$ and $\log_LM$
  actually only depend on the intersection of $U$ with
  $\cF$. Therefore, one may as well regard $U$ as an open subset of
  $\cF$ which can be extended to an open subset of $\cA$ if needed.
\end{remark}

\begin{lemma}\label{lemma:genericstructure}
  For a connected, locally convex \qH\ $K\tau$-manifold let
  $\cF\subseteq\cA$ be the smallest face of $\cA$ containing $\cP_M$
  and let $U\subseteq\cF$ be its relative interior. Then
  $L_M:=K_{a\tau}$ is independent of $a\in U$. The preimage
  $M_0:=m^{-1}(L\tau)\cap m_+^{-1}(U)$ is a connected \qH\
  $L_M$-manifold such that $K\times^{L_M}M_0\to M$ is a diffeomorphism
  onto an open dense subset of $M$. Let $L_M'\subseteq L_M$ be the
  kernel of $L_M\to\Aut(M_0)$. Then $A_M=L_M/L_M'$ is a torus acting
  freely on a dense open subset of $M_0$. In particular, $L_M'$ is a
  generic isotropy group for the $K$-action on $M$.
\end{lemma}

\begin{proof}
  All assertions are local over $\cP_M$. \cref{thm:localstructure}
  implies therefore that we may assume that $M$ is Hamiltonian. But in
  this context, all assertions have been proved in \cite{LMTW}.
\end{proof}

In practice, it is more convenient to encode $L_M'$ by the character
group $\Lambda_M:=\Xi(A_M)$. Let $A\subseteq L_M$ be the subtorus with
$\Lie A=\fa=\ft^\tau$. Then $A\to A_M$ is surjective and we can
consider $\Lambda_M$ as a subgroup of $\Xi(A)\otimes\RR=\fa^*=\fa$.

The two objects $\cP_M$ and $\Lambda_M$ are not unrelated: Let
$\fa_M\subseteq\fa$ be the affine span of $\cP_M$. Since
$\fa_{M_0}=\fa_M-a=\Vfa_M$ and $\Lambda_{M_0}=\Lambda_M$ we see that
$\Lambda_M$ is a lattice inside $\Vfa_M$. Its dimension will be called
the \emph{rank} $\|rk|M$ of $M$.

\section{Multiplicity free manifolds}

We can use \cref{lemma:genericstructure} to compute the dimension of
$M$ where $M$ is connected, locally convex and \qH. Since $m_+$ is
$K$-invariant we get an induced map $M/K\to\cP_M$ which is surjective by
definition. This implies the inequality
\[\label{eq:dimPM}
  \dim M/K\ge\dim\cP_M=\rk M.
\]
On the other side, the generic isotropy group $L_M'$ has dimension
$\dim L_M-\dim A_M=\dim L_M-\rk M$. Thus \eqref{eq:dimPM} is
equivalent to
\[\label{eq:dimMdimK}
  \dim M\ge\dim K-\dim L_M+2\rk M.
\]
This suggests the following

\begin{definition}

  A \qH\ manifold $M$ is \emph{multiplicity free} if it is connected,
  locally convex, and the equivalent inequalities \eqref{eq:dimPM} and
  \eqref{eq:dimMdimK} are actually equalities.
\end{definition}

Observe, that this is a verbatim generalization of the concept of
multiplicity freeness of ordinary Hamiltonian manifolds (see, e.g.,
\cite{KnopAutoHam}*{Def.\ 2.1}). As in that case, multiplicity
freeness for \qH\ manifolds has numerous equivalent
characterizations. The following statement mentions a couple of
them. Thereby we adopt the convention that a submanifold
$N\subseteq M$ is called \emph{coisotropic} if
$(T_xN)^\perp\subseteq T_xN$ for all $x\in N$ even in the case that
the $2$-form $\omega$ is degenerate.

\begin{proposition}\label{prop:equivMF}
  For a connected, locally convex \qH\ manifold $M$ the following are
  equivalent:

  \begin{enumerate}

  \item\label{it:mf2} $M$ is multiplicity free.

  \item\label{it:mf1} The induced map $m_+/K:M/K\to \cP_M$ is a
    homeomorphism

  \item\label{it:mf3} $\fk_{m(x)}x=\ker D_xm$ for $x$ in a non-empty
    open subset of $M$.

  \item\label{it:mf4} The orbit $Kx$ is coisotropic for $x$ in a
    non-empty open subset of $M$.

  \item\label{it:mf5} $\log_LM$ is a multiplicity free Hamiltonian
    manifold for all (or one) open subset(s) $U$ as in
    \cref{thm:localstructure} with $U\cap\cP_M$ connected.

  \end{enumerate}

\end{proposition}

\begin{proof}

  \ref{it:mf2}$\Leftrightarrow$\ref{it:mf5} Let $M_0:=\log_LM$ which
  is connected by \cref{lemma:conv}\ref{it:conv1}. Since there are
  open embeddings $M_0/L\into M/K$ and $\cP_{M_0}\into\cP_M-a$ we get
  that $\dim M/K=\dim\cP_M$ if and only if $\dim M_0/L=\dim\cP_{M_0}$.

  \ref{it:mf2}$\Leftrightarrow$\ref{it:mf1} The manifold $M$ is by
  definition multiplicity free if and only if $\dim
  M/K=\dim\cP_M$. Since the fibers of $m_+$ are connected, this is
  equivalent to $m_+/K$ being bijective over a non-empty open subset
  $U$ of $\cP_M$. Now it suffices to prove that $m_+/K$ is (globally)
  injective since $m_+/K$ is already surjective and open. Using
  \cref{thm:localstructure} we may assume that $M$ is actually
  Hamiltonian. Let $V\subseteq\cP_M$ the non-empty set of points in
  which $m_+/K$ is locally invertible. It follows from the symplectic
  slice theorem (see e.g. \cite{GS}*{2.3}) that $m_+/K$ is locally
  analytic (even semialgebraic). It follows that $(m_+/K)^{-1}(V)$ is open and
  dense in $M/K$. Since $M/K$ is connected it suffices to show that
  $V$ is closed. So let $a\in\cP_M$, $b\in F:=(m_+/K)^{-1}(a)$ and
  $a_i\in V$ a sequence so that the preimages $b_i\in M/K$ converge to
  $b$. Consider the open subset
  $U:=M/K\setminus\{b_i\mid i\}\cup\{b\}$. Then the image $(m_+/K)(U)$
  is an open subset of $\cP$ which does not contain the $a_i$. Thus it
  does not contain $a$ either, i.e., $m_+/K$ is injective over $a$.

  \ref{it:mf2}$\Leftrightarrow$\ref{it:mf3}: It follows from
  \cref{lemma:genericstructure} that the generic orbit in $m(M)$ is
  isomorphic to $K/L_M$ and that $\dim m(M)/K=\dim\cP=\rk M$. Thus
  \[
    \dim m(M)=\dim K-\dim L_M+\rk M.
  \]
  Thus, $M$ is multiplicity free if and only if $\dim m(M)=\dim-\rk M$
  or, equivalently, $\dim F_x=\rk M$ where
  $F_x=m^{-1}(m(x))\subseteq M_0$ is a generic fiber of $m$. On the
  other side, the isotropy group $K_{m(x)}$ is conjugate to $L_M$
  which means that the generic $K_{m(x)}$-orbits in $F_x$ have dimension $\rk
  M$. Thus $M$ is multiplicity free if and only if $K_{m(x)}$ has open
  orbits in $F_x$ which is precisely \ref{it:mf3}.
  
  \ref{it:mf3}$\Leftrightarrow$\ref{it:mf4}: Condition \ref{it:mf3} is
  equivalent to $\ker D_xm\subseteq\fk x$. Hence the equivalence with
  \ref{it:mf4} follows from \cref{lemma:mf1}.
\end{proof}

\begin{remark}
  Generally, the fiber of $m_+/K$ over $a\in\cP_M$ is called the
  \emph{symplectic reduction of $M$ in $a$}. So, another
  characterization of multiplicity freeness of a locally convex,
  connected $M$ is that all its symplectic reductions are points.
\end{remark}

  \section{Local models}

In this section we make the first step towards the goal of classifying
multiplicity free manifolds. More precisely, we characterize the pairs
$(\cP,\Lambda)$ which eventually will be shown to be those of the form
$(\cP_M,\Lambda_M)$ for $M$ multiplicity free.

Let $G=K^\CC$ be the complexification of $K$ which is a connected
complex reductive group. Let $B\subseteq G$ be a Borel subgroup
containing a maximal torus $T\subseteq K$ and let
$\Xi^\CC(B):=\Hom(B,\CC^*)\cong\ZZ^{\operatorname{rk}G}$ be its
algebraic character group. Since $\Xi^\CC(B)$ can be identified with
$\Xi(T)$ the space $\Xi^\CC(B)\otimes\RR$ identifies with $\ft^*$ and
therefore with $\ft$. The choice of $B$ determines a Weyl chamber
$\ft^+\subseteq\ft$ and characters lying in $\ft^+$ are called
dominant. Recall, that there is a $1:1$-correspondence
$\chi\mapsto L_\chi$ between dominant characters and irreducible
representations of $G$.

An algebraic $G$-variety $X$ is called \emph{spherical} if $B$ has a
Zariski dense orbit in $X$. In the following we are only interested in
the case when $X$ is affine. Then there is a purely representation
theoretic criterion for sphericity due to
Vinberg-Kimel{\cprime}fel{\cprime}d, \cite{VK}: Let $\CC[X]$ be the
ring of regular functions on $X$. Then $X$ is spherical if and only if
$\CC[X]$ is multiplicity free, i.e., it is a direct sum of
\emph{distinct} irreducible representations of $G$. Thus, for a
spherical variety there is a well defined set
$\Lambda_X^+\subseteq\ft^+$ of dominant weights such that
\[
  \CC[X]\cong\bigoplus_{\chi\in\Lambda_X^+}L_\chi.
\]

\begin{definition}
  For an affine spherical $G$-variety $X$ let
  $\cP_X:=\RR_{\ge0}\Lambda^+$ be the convex cone and let
  $\Lambda_X:=\ZZ\Lambda_X^+$ be the subgroup spanned by
  $\Lambda_X^+$ in $\ft^*$. Then $(\cP_X,\Lambda_X)$ will be called the
  \emph{spherical pair} determined by $X$.
\end{definition}

Recall the notation of \cref{thm:conjugacyclasses} and in particular
that every $a\in\cA$ determines a subgroup $K_{a\tau}\subseteq
K$. Because $\fa\subseteq\fk_{a\tau}$ is a Cartan subalgebra, the
characters of the complexification $K_{a\tau}^\CC$ can be considered
to be elements of $\fa$. Thus, if $X$ is a spherical
$K_{a\tau}^\CC$-variety then $\cP_X$ and $\Lambda_X^+$ will be subsets
of $\fa$.

\begin{definition}\label{def:spherical}

  Let $\cP\subseteq\cA$ be a subset and let $\Lambda\subseteq\fa$ be a
  subgroup.

  \begin{enumerate}

  \item\label{it:spherical1} The pair $(\cP,\Lambda)$ is
    \emph{spherical in $a\in\cP$} if there is a \emph{smooth} affine
    spherical $K_{a\tau}^\CC$-variety $X$ and a neighborhood $U$ of
    $a$ in $\fa$ such that
    \[\label{eq:spherical1}
      ((\cP_X+a)\cap U,\Lambda_X)=(\cP\cap U,\Lambda).
    \]

  \item The pair $(\cP,\Lambda)$ is \emph{spherical} if $\cP$ is
    connected and $(\cP,\Lambda)$ is spherical in all $a\in\cP$.

  \item The pair $(\cP,\Lambda)$ is \emph{convex} if $\cP$ is convex.

  \end{enumerate}

\end{definition}

First we remark that sphericity is a necessary condition:

\begin{lemma}\label{lemma:sphericalimage}
  Let $M$ be a (convex and) \mf\ $K\tau$-manifold. Then
  $(\cP_M,\Lambda_M)$ is (convex and) spherical.
\end{lemma}

\begin{proof}
  Using the Local Structure \cref{thm:localstructure}, the assertion
  is reduced to the Hamiltonian case where it is follows from
  \cite{KnopAutoHam}*{Thm.~11.2}.
\end{proof}

The following remarks are not necessary to understand the proof of the
classification theorem but they are useful for recognizing and
constructing spherical pairs.

\begin{Remarks}\label{rem:localmodel}
  Let $(\cP,\Lambda)$ be a spherical pair.

  \Item Let $\fa_\cP\subseteq\fa$ be the affine subspace spanned by
  $\cP$. Then the group $\Lambda$ is a lattice in the group of
  translations $\Vfa_\cP$ which in turn follows from
  \eqref{eq:spherical1} and the fact that $\cP_X$ and $\Lambda_X$ have
  the same $\RR$-span.

  \Item The subset $\cP$ is locally polyhedral and therefore locally
  convex, locally closed and solid inside $\fa_\cP$. It follows, in
  particular, that $\cP$ has a well defined dimension, namely
  $\dim\cP:=\dim\fa_\cP=\|rk|\Lambda$. This follows from
  \eqref{eq:spherical1}, the fact that $\cP_X$ is a finitely generated
  cone and the connectedness of $\cP$.

  \Item The \emph{tangent cone} $C_a\cP$ of $\cP$ in $a$ is defined as
  the convex cone generated in $\fa$ by $(\cP\cap U)-a$ where
  $\cP\cap U$ is a convex neighborhood of $a$ in $\cP$. It is easy to
  see that $C_a\cP$ is independent of the choice of $U$. From
  \cref{def:spherical} follows
  \[
    (\cP_X,\Lambda_X)=(C_a\cP,\Lambda).
  \]
  which means that $(\cP_X,\Lambda_X)$ is uniquely determined by
  $\cP$, $\Lambda$ and $a$.

  \Item For any affine spherical variety $X$ the set $\Lambda_X^+$ is
  a monoid, i.e., is additively closed. The reason for this is that
  $\CC[X]$ is a domain. If $X$ is normal (e.g., smooth) then
  $\Lambda_X^+=\cP_X\cap\Lambda_X$. This means that $\Lambda_X^+$ and
  $(\cP_X,\Lambda_X)$ determine each other.

  \Item It is a non-trivial theorem of Losev
  \cite{LosevKnopConj}*{Thm.\ 1.3} that a smooth affine spherical
  $G$-variety is uniquely determined by its weight monoid
  $\Lambda_X^+$. With \emph{c)} and \emph{d)} this implies that the
  $K_{a\tau}^\CC$-variety $X$ from \cref{def:spherical} is uniquely
  determined by $\cP,\Lambda$ and $a$. We call $X$ the \emph{local
    model} of $(\cP,\Lambda)$ in $a$.

  \Item Assume $(\cP,\Lambda)=(\cP_M,\Lambda_M)$ for some \mf\
  manifold $M$ and let $X$ be the local model in $a\in\cP$. We sketch
  how a neighborhood of $m_+^{-1}(a)$ in $M$ is described by
  $X$. Since $X$ is affine, it can be embedded equivariantly into a
  $K_{a\tau}^\CC$-vector space $Z$ as a closed
  $K_{a\tau}^\CC$-subvariety. The choice of a $K_{a\tau}$-invariant
  Hermitian scalar product on $Z$ induces a Kähler- hence symplectic
  structure on $X$. With
  $\tilde m(x):=(\xi\mapsto\half\<\xi x,x\>)\in \fk_{a\tau}^*$, the
  variety $X$ becomes a Hamiltonian $K_{a\tau}$-manifold and
  $(\cP_X,\Lambda_X)$ is also the pair attached to $X$ when considered
  as a Hamiltonian manifold
    (see \cite{Sjamaar}*{Thm.\ 4.9} or \cite{LosevKnopConj}*{Prop.\ 8.6(3)}).
%  (see \cite{BrionMoment}*{Prop.\ 5.2} or \cite{Bart2}*{Prop.\ 2.2}).
  Moreover, there is an open neighborhood
    $U$ of $a$ in $\fa$ such that
    \newcommand{\xysubseteq}{\ar@{}[r]|>>>{\displaystyle\subseteq}}
  \[
    \cxymatrix{
      K\times^{K_{a\tau}}\tilde m_+^{-1}(\cP_X\cap(U-a))\ \ar[d]_\sim\ar@{^(->}[rr]^>>>>>>>{\text{open}}&&\ K\times^{K_{a\tau}}X\\
      m_+^{-1}(\cP\cap U)\ \ar@{^(->}[rr]^>>>>>>>>>>>>{\text{open}}&&\ M}
  \]

  \Item Affine spherical varieties have been classified by Knop-Van
  Steirteghem in \cite{KVS}. Pezzini-Van Steirteghem, \cite{PVS},
  developed an algorithm for deciding the sphericity of a pair
  $(\cP,\Lambda)$ in any given point $a\in\cP$.

  \Item For deciding sphericity in all points of $\cP$ the following
  observations help:

\begin{lemma}
  The set of points $a\in\cP$ where $(\cP,\Lambda)$ is spherical is
  open in $\cP$.
\end{lemma}

\begin{proof}
  Let $a$, $U$ and $X$ as in
  \cref{def:spherical}\ref{it:spherical1}. Recall from
  Remark~\ref{rem:localmodel} \emph{f)} that $X$ carries a structure
  as multiplicity free Hamiltonian $K$-manifold with pair
  $(\cP_X,\Lambda_X)$. Now it follows from \cite{KnopAutoHam}*{Thm.~11.2}
  applied to $X$ that every point of $\cP\cap U=(\cP_X+a)\cap U$ is
  spherical.
\end{proof}

The tangent cone is constant on open intervals:

\begin{lemma}\label{lemma:interval}
  Let $\cP\subseteq\fa$ be a locally convex subset and let
  $I\subseteq\cP$ be an open interval. Then the tangent cone $C_a\cP$
  is the same for all $a\in I$.
\end{lemma}

\begin{proof}
  It suffices to show that $a\mapsto C_a\cP$ is locally constant on
  $I$. Fix $a\in I$. Since the tangent cone depends only on a
  neighborhood of $a$ we can replace $\cP$ by a convex neighborhood
  $\cP\cap U$ and therefore assume that $\cP$ is convex. Now let
  $b\in I$ be so close to $a$ that the two points $b':=b+(b-a)$
  and $a':=a+(a-b)$ still lie in $I$. Then
  \[\begin{split}
      C_a\cP=\RR_{\ge0}(\cP-a)&=\RR_{\ge0}(\cP-b+(b'-b))\\
      &\subseteq\RR_{\ge0}(\cP-b)+\RR_{\ge0}(b'-b)\subseteq C_b\cP
    \end{split}\] and similarly $C_b\cP\subseteq C_a\cP$.
\end{proof}

Now for compact sets $\cP$ we have the following criterion:

\begin{proposition}\label{prop:polyhedron}
  Let $\cP\subseteq\cA$ be compact. Then a pair $(\cP,\Lambda)$ is
  spherical if and only if $\cP$ is a polytope and the pair
  $(\cP,\Lambda)$ is spherical in every vertex of $\cP$.
\end{proposition}

\begin{proof} Assume first that $(\cP,\Lambda)$ is spherical. Then
  $\cP$ is a polytope by \cref{lemma:polyhedron} below.

  Now assume conversely that $\cP$ is a polytope such that
  $(\cP,\Lambda)$ is spherical in every extremal point. Then it follows
  from the preceding two lemmas that the set of spherical points is
  convex. Since it contains all extremal points it is all of $\cP$.
\end{proof}

\end{Remarks}

\begin{lemma}\label{lemma:polyhedron}
  Let $\fa$ be an affine space and let $\cP\subseteq\fa$ be locally
  polyhedral, compact, and connected. Then $\cP$ is a polytope.
\end{lemma}

\begin{proof}
  Since $\cP$ is closed, connected, and locally convex it is convex
  (Tietze \cite{Tietze}*{Satz~1}). It follows that $\cP$ is the convex hull of
  its extremal points (Krein–Milman). But there are only finitely many
  extremal points since $\cP$ is locally polyhedral. So $\cP$ is a
  polytope.
\end{proof}

\section{Classification of \mf\ manifolds}\label{sec:class}

Now we can tie all strings together and prove our main theorem.

\begin{theorem}\label{thm:main}

  Let $K$ be a simply connected compact Lie group with twist $\tau$
  and let $\cA\subseteq\fa$ be as in \cref{thm:conjugacyclasses}. Let
  $\cP\subseteq\cA$ be a subset and $\Lambda\subseteq\fa$ a
  subgroup. Then there is a convex \mf\ $K\tau$-manifold $M$ with
  $(\cP_M,\Lambda_M)=(\cP,\Lambda)$ if and only if $(\cP,\Lambda)$ is
  convex and spherical. Moreover, this $M$ is unique up to
  isomorphism.

\end{theorem}

For compact \mf\ manifolds this means:

\begin{corollary}

  Compact \mf\ $K\tau$-manifolds are classified by pairs
  $(\cP,\Lambda)$ for which $\cP\subseteq\cA$ is a polytope and
  $(\cP,\Lambda)$ is spherical in every vertex of $\cP$.

\end{corollary}

\begin{proof}
  Follows from \cref{thm:main} using \cref{lemma:conv}\ref{it:conv4}
  and \cref{prop:polyhedron}.
\end{proof}

We reduce the proof of the main \cref{thm:main} to a statement about
automorphisms. For this we use the language of \emph{gerbes}.

Let $(\cP,\Lambda)$ be a spherical pair. For any open subset
$U\subseteq\cP$ let $\sM_{\cP,\Lambda}(U)$ be the category whose
objects are \mf\ $K\tau$-manifolds $M$ with
$(\cP_M,\Lambda_M)=(U,\Lambda)$. The morphisms are $K$-equivariant
diffeomorphisms $\phi:M\to M'$ such that $\omega=\phi^*\omega'$ and
that $m=m'\circ\phi$. Since all morphisms are invertible,
$\sM_{\cP,\Lambda}(U)$ is a \emph{groupoid}.

For any pair of open subsets $U\subseteq V\subseteq\cP$ the groupoids
are linked by restriction functors
$\res_U^V:\sM_{\cP,\Lambda}(V)\to\sM_{\cP,\Lambda}(U): M\mapsto
m^{-1}(U)$ which satisfy $\res_U^V\circ\res_V^W=\res_U^W$
whenever $U\subseteq V\subseteq W$.

This means that $\sM_{\cP,\Lambda}$ is a \emph{presheaf of groupoids}
over $\cP$. Since all morphisms and objects can be glued along any
gluing data, the system of categories $\sM_{\cP,\Lambda}$ is even a
\emph{sheaf of groupoids}, also known as \emph{stack} (see
e.g. \cite{Brylinski}*{Def.\ 5.2.1} for a precise definition).  A very
particular kind of stacks are \emph{gerbes} which means that they have
the following two additional properties:

  \begin{enumerate}

  \item $\sM$ is locally non-empty, i.e. every point $a\in\cP$ has an
    open neighborhood $U\subseteq\cP$ such $\sM(U)\ne\leer$ and

  \item any two objects $M,M'\in\sM(V)$ with $V\subseteq\cP$ open are
    locally isomorphic, i.e., every $a\in V$ has an open neighborhood
    $U\subseteq V$ such that $\res^V_UM\cong\res^V_UM'$.

\end{enumerate}

See e.g. \cite{Brylinski}*{Def.\ 5.2.4} (where the definition of a
gerbe is combined with that of a band) or
\cite{stacks23}*{Def.~8.11.2}. In the next two proofs we
extensively use the local structure theorem \ref{thm:localstructure} and
the obvious fact that if it holds for an open set $U$ then it will
also hold for all smaller open subsets $V$ regardless of whether $V$
contains the base point $a$ or not.

\begin{theorem}
  Let $(\cP,\Lambda)$ be spherical. Then $\sM_{\cP,\Lambda}$ is a
  gerbe over $\cP$.
\end{theorem}

\begin{proof}
  We have to prove that \emph{a)} and \emph{b)} hold for
  $\sM_{\cP,\Lambda}$.

  Hereby, \emph{a)} is basically the definition of a spherical pair:
  let $X$ be a smooth affine spherical variety with
  \eqref{eq:spherical1}. Then $X$ has a $K$-Hamiltonian structure such
  that $(\cP_X,\Lambda_X)=(C_a\cP,\Lambda)$ (cf.\ remark
  \ref{rem:localmodel} \emph{f)}). Choosing $U$ small enough as in
  \cref{thm:localstructure}, there is a \qH\ manifold $M$ with
  $m_+(M)=(a+C_aP)\cap U$ and $\log_LM=X_{U-a}$. Then $M$ has the
  required properties.

  The second assertion \emph{b)} follows using
  \cref{thm:localstructure} from \cite{Losev}*{Thm. 1.3} to the effect
  that the local model $X$ is uniquely determined by $\Lambda_X^+$
  (cf.\ \ref{rem:localmodel} \emph{e)} and \cite{KnopAutoHam}*{Thm.\
    2.4}).
\end{proof}

A particular nice type of gerbes are those for which the automorphism
group of every object is abelian. In this case, the automorphism
groups combine to a sheaf of abelian groups $\fL_\sM$ on $\cP$, the
so-called \emph{band} of $\sM$. More precisely, let $M, M'\in\sM(U)$
be two objects over $U$ and $\phi:M\overset\sim\to M'$ an
isomorphisms. Then $\phi$ induces an isomorphism
$\tilde\phi:\Aut(M)\overset\sim\to\Aut(M')$ by
$\tilde\phi(f)=\phi f\phi^{-1}$. If $\psi:M\overset\sim\to M'$ is
another isomorphism then one checks easily that
$\tilde\psi^{-1}\tilde\phi\in \Aut(\Aut(M))$ is conjugation by
$\psi^{-1}\phi\in\Aut(M)$. Thus, if $\Aut(M)$ is abelian, then
$\Aut(M)$ depends only on the isomorphism class of $M$. This means,
that $\fL_\sM^\#(U):=\Aut(M)$ with $M\in\sM(U)$ is a well-defined
presheaf of abelian groups on $\cP$. The band $\fL_\sM$ is by
definition the sheafification of $\fL_\sM^\#$. See
\cite{stacks23}*{Lemma~8.11.8} for details.

In the remainder of this section we are going to finally make use of
the results of first part of this paper. Recall, in particular, the
sheaf of abelian groups $\fL_{\cP,\Lambda}^{\Phi(*)}$ from
\cref{def:fL}.

\begin{theorem}\label{thm:gerbe}
  Let $(\cP,\Lambda)$ be spherical. Then the gerbe
  $\sM:=\sM_{\cP,\Lambda}$ has abelian automorphism groups. Its band
  $\fL_\sM$ is canonically isomorphic to $\fL_{\cP,\Lambda}^{\Phi(*)}$
  for a unique local root system on $\cP\subseteq\fa_\cP$.
\end{theorem}

\begin{proof} For every $a\in\cP$ choose a open neighborhood $V$ as in
  \cref{cor:LocStr} and let $\overline\sM$ be the gerbe of Hamiltonian
  manifolds over $U:=\cP\cap V$. Then the Local Structure
  \cref{thm:localstructure} yields an isomorphism of gerbes
  $\overline\sM\overset\sim\to\sM|_U$. The automorphism groups of the
  objects of $\overline\sM$ have been determined in
  \cite{KnopAutoHam}*{Thm.~9.2}. Translated into the language of
  gerbes the result is that there is unique root system $\Phi(a)$ with
  $\alpha(a)=0$ for all $\alpha\in\Phi(a)$ such that the band of
  $\overline\sM$ is isomorphic to $\fL_{U,\Lambda}^{\Phi(a)}$. Now it
  follows from \cite{KnopAutoHam}*{eq.\ (9.4)} that the system
  $((\Phi(a))_{a\in\cP},\Lambda)$ forms a local root system on
  $\cP$. In other words, the band $\fL_\sM$ is locally isomorphic to
  $\fL:=\fL_{\cP,\Lambda}^{\Phi(*)}$.
  
  We claim that that the local isomorphisms
  $\Phi_U:\fL|_U\to\fL_\sM|_U$ glue to a global isomorphism
  $\Phi:\fL\overset\sim\to\fL_\sM$. This is not completely obvious
  since the local model of $(\cP,\Lambda)$ at $a\in\cP$ is a
  Hamiltonian manifold for the group $K_{a\tau}$ which therefore
  does not depend continuously on $a$.

  To bypass this problem we restrict the isomorphisms $\Phi_U$ to
  $\cP^0$, the interior of $\cP$ inside $\fa_\cP$. Since $\cP^0$ is
  dense in $\cP$ and since the sections of both $\fL$ and $\fL_\sM$
  are continuous it suffices to prove compatibility on $\cP^0$.

  To this end, let $a\in\cP$ and let $U\subseteq\cP$ be a convex
  neighborhood of $a$ in $\cP$. Then $U^0:=U\cap\cP^0$ is open in
  $\cP$ and convex, as well. Choose $U$ small enough such that an
  object $M\in\sM(U)$ exists. Let $L:=K_{a\tau}$. Then we can choose
  $U$ such that also $\Mq:=\log_LM$ exists. With
  $M^0:=M_{U^0}=m_+^{-1}(U^0)$ and
  $\Mq^0:=\Mq_{U^0}=\overline m_+^{-1}(U^0)$ we obtain the following
  diagram:
  \[
    \dxymatrix{\fL(U^0)\ar[d]_\sim&
      C^\infty(U^0)\ar@{>>}[l]_\epsilon\ar[dl]_{\overline h}\ar[d]^h\\
      \Aut(\Mq^0)\ar[r]^\sim&\Aut(M^0).}
  \]
  Here $\epsilon$ is the map $f\mapsto\exp(\nabla f)$ from
  \eqref{eq:defepsilon} which is surjective by \cref{lemma:nabla}. The
  maps marked with $h$ and $\overline h$ map $f$ to the Hamiltonian
  flow $\exp(H_f)$ on $M^0$ and $\Mq^0$, respectively (see
  \cite{AMM}*{Prop.~4.6} for Hamiltonian flows on \qH\ manifolds). The
  upper left triangle is commutative by \cite{KnopAutoHam}*{Thm.\
    9.1}. One can easily check that the Hamiltonian flow on $M$
  restricts to the corresponding Hamiltonian flow on $\Mq$. So the
  bottom right triangle commutes as well. It follows that the upper
  right triangle commutes. Because of $\Aut(M^0)=\fL_\sM(U^0)$ we
  obtain the commutative triangle
    \[
      \dxymatrix{\fL(U^0)\ar[dr]_{\Phi_U|_{U_0}}&
        C^\infty(U^0)\ar@{>>}[l]_\eta\ar[d]^h\\
     &\fL_\sM(U^0).}
 \]
 Since both $\eta$ and $h$ depend only on $U^0$ (instead of $U$), the
 same holds for the restriction $\Phi_U|_{U_0}$. This shows that all
 isomorphisms $\Phi_U$ coincide on the intersection of their
 domains. Hence they glue to a global isomorphism $\Phi$.
\end{proof}

If $\cP$ is convex we can say more:

\begin{corollary}
  Let $(\cP,\Lambda)$ be a spherical pair with $\cP$ convex. Then the
  higher cohomology of the band of $\sM_{\cP,\Lambda}$ vanishes.
\end{corollary}

\begin{proof}
  It follows from \cref{prop:TrivialRootCrit} that the local root
  system $\Phi(*)$ is trivial. Indeed, every convex set is solid in
  its affine span. Moreover, it follows from
  \cite{KnopAutoHam}*{Thm.~4.1} that every local Weyl group $W(a)$ is
  a subquotient of the Weyl group of $K_{a\tau}$. Thus every element
  $w\in W(a)$ can be lifted to an element $\tilde w$ of
  $W_{\Phi_\tau}$ with $\tilde w\fa_\cP=\fa_\cP$. Thus, the second
  condition follows from $\cP\subseteq\cA$ and the fact that every
  $W_{\Phi_\tau}$-orbit meets $\cA$ in exactly one point.

  Now the assertion follows from \cref{T1}.
\end{proof}

\begin{proof}[Proof of \cref{thm:main}]
  Let $\sM:=\sM_{\cP,\Lambda}$ and let $M_1,M_2\in\|Obj|\sM(\cP)$ be
  two global objects. Then the sheaf $\fI:=\|Isom|_\cP(M_1,M_2)$ is a
  torsor for $\fL_\sM$. The vanishing of $H^1(\cP,\fL_\sM)$ implies
  that $\fI$ has a global section, i.e., $M_1$ and $M_2$ are
  isomorphic (see \cite{Brylinski}*{5.2.5(1)}). This shows uniqueness
  of $M$.

  For the existence, observe that there are arbitrary fine open
  coverings $\cP=\bigcup_\mu U_\mu$ which are good in the sense that
  $H^1$ of $\fL_\sM$ vanishes on all $U_\mu$ and $U_\mu\cap
  U_\nu$. Indeed one can take for the $U_\mu$ intersections of $\cP$
  with a small open balls. Then all $U_\mu$ and $U_\mu\cap U_\nu$ are
  convex, hence have vanishing $H^1$. Under these circumstances the
  vanishing of $H^2(\cP,\fL)$ implies that $\sM$ is the only gerbe
  with band $\fL_\sM$ (\cite{Brylinski}*{5.2.8}, see also
  \cite{stacks23}*{Lemma 21.11.1}). Thus $\sM$ is isomorphic to the category
  of all $\fL_\sM$-torsors which has a global object, namely the
  trivial torsor. So $\sM(\cP)\ne\leer$.
\end{proof}

\section{Examples}\label{sec:examples}

We conclude this paper with a series of examples. It should be
mentioned that Paulus has obtained many more in his
thesis \cite{Paulus}.

\subsection*{Doubles}

A particularly important \qH\ manifold is the \emph{double} $D(K)$ of
a Lie group $K$. It was defined in \cite{AMM} and since it was used
for the proof of \cref{thm:localstructure}. Hence, our construction is
just an a posteriori reason for the existence of $D(K)$. In case $K$
is compact and simply connected the double has a nice description in
terms of a spherical pair: Recall from the proof of
\cref{thm:localstructure} that the acting group is $\Kq=K\times K$. As
a manifold $D(K)$ equals $K\times K$ with $\Kq$ acting on $D(K)$ as
\[
  (x,y)*(a,b)=(xay^{-1},xby^{-1}).
\]
The momentum map is
\[\label{eq:double2}
  m(a,b)=(ab^{-1},a^{-1}b).
\]
(this differs from \cite{AMM}*{\S3.2} by the coordinate change
$(a,b)\mapsto(a,b^{-1})$ on $D(K)$).  Let $\cA$ and $\Lambda$ be the
alcove and the weight lattice of $K$. Then $\overline\cA=\cA\times\cA$
and $\overline\Lambda=\Lambda\oplus\Lambda$ are alcove and weight
lattice of $\Kq$. Let $w_0$ the longest element of the Weyl group $W$
of $K$ and $\delta=\id\times(-w_0):\ft\to\ft\oplus\ft$. Then
\[
  (\cP_{D(K)},\Lambda_{D(K)})=
  (\delta(\cA),\delta(\Lambda))\subseteq
  (\overline\cA,\overline\Lambda).
\]
Indeed, let $T\subseteq K$ be a maximal torus and
$(a,b)\in\overline T:=T\times T$. Then \eqref{eq:double2} shows that
$\cP_{D(K)}$ is the set of $(a_1,a_2)\in\overline\cA$ such that
$\exp(a_1)$ is the $w_0$-conjugate of $\exp(-a_2)$. This shows
$\cP_{D(K)}=\delta(\cA)$. Furthermore, for generic $(a,b)$ the
stabilizer of $m(a,b)$ is $\overline T$ which shows
$L_{D(K)}=\overline T$. The stabilizer of $(a,b)$ in $L_{D(K)}$ is the
diagonal torus $\Delta T$. Thus $A_{D(K)}=\overline T/\Delta T\cong T$
which implies $\Lambda_{D(K)}=\delta(\Lambda)$.

\begin{remark}

  The case of doubles shows that the classification of \qH\ manifolds
  does depend on the choice of an invariant scalar product on
  $\fk$. To see this observe that the alcove $\overline\cA$ for
  $\Kq=\SU(2)\times \SU(2)$ is a rectangle whose side lengths depend
  on the chosen metric. The double $D(\SU(2))$ corresponds to the case
  when $\cP$ is the diagonal of $\cA$. In order for $(\cP,\Lambda)$ to
  be spherical, $\cP$ has to be parallel to the sum $\alpha+\alpha'$
  of the simple roots of $\Kq$. This holds if and only if
  $\overline\cA$ is a square, i.e., when the metrics on both factors
  of $\Kq$ are the same. This phenomenon does not occur when $K$ is
  simple or for Hamiltonian manifolds.

\end{remark}

\subsection*{Groups of rank $1$}

Let $K=\SU(2)$. Then $\cA$ is an interval and $\cP_M\subseteq\cA$ is a
subinterval. If $\cP_M\ne\cA$ then $M$ is of the form $K\times^LM_0$
(see \cref{thm:localstructure}) for some Hamiltonian $L$-manifold
$M_0$ with $L\subseteq K$. Quasi-Hamiltonian manifolds which are not
of this form will be called \emph{genuine}. Since genuine \mf\
$\SU(2)$-manifolds have necessarily $\cP_M=\cA$ we just have to check
which lattice $\Lambda_M$ can occur. Because the possible local models
in the end points are the $\SL(2,\CC)$-varieties $\CC^2$,
$\SL(2,\CC)/\CC^*$ and $\SL(2,\CC)/N(\CC^*)$ we get $3$ different
genuine \mf\ $\SU(2)$-manifolds:

\begin{itemize}

\item $\Lambda_M=P\cong\ZZ\omega$, the weight lattice of
  $\SU(2)$. Here $M$ is obtained by equivariantly gluing two copies of
  the closed unit disk $D$ in $\CC^2$ along their boundary $S^3$. One
  can check that the result of such a glueing is always diffeomorphic
  to the $4$-sphere $S^4$. This example has been found by
  Alekseev-Meinrenken-Woodward \cite{AMW} under the name \8spinning
  $4$-sphere\9.

\item $\Lambda_M=2P$. In this case one can show that
  $M\cong \P^1(\CC)\times \P^1(\CC)$.

\item $\Lambda_M=4P$. Here, $M$ is the quotient of the
  previous case by the switching involution. Hence $M\cong \P^2(\CC)$.

\end{itemize}

There is another affine root system of rank $1$, namely
$\sA_2^{(2)}$. It is the root system of $K=\SU(3)$ with the twist
being an outer automorphism $\tau$ of $K$, e.g., complex
conjugation. The alcove $\cA$ is an interval and the two simple roots
$\alpha_0$, $\alpha_1$ satisfy $\Valpha_0=-2\Valpha_1$. The weight
lattice of the affine root system is $P=\ZZ\Valpha_1$. The
centralizers corresponding to the end points are $\SU(2)$ and
$\SO(3)$, respectively. Let $\cP_M=\cA$. Then a discussion as above
yields two cases

\begin{itemize}

\item $\Lambda_M=P$: In this case, the local models are $\CC^2$ and
  $\SO(3,\CC)/\SO(2,\CC)$.

\item $\Lambda_M=2P$. In this case, the local models are
  $\SL(2,\CC)/\SO(2,\CC)$ and $\SO(3,\CC)/\OG(2,\CC)$.

\end{itemize}

Note that $\Lambda_M=4P$ does not work since
$\Lambda_X^+=\ZZ_{\ge0}(4\Valpha_1)$ is not the weight monoid of any
\emph{smooth} affine spherical $\SO(3,\CC)$-variety.

\subsection*{Manifolds of rank $1$}

The spinning $4$-sphere has been generalized by
Hurtubise-Jeffrey-Sjamaar in \cite{HJS} to that of a spinning
$2n$-sphere. In our terms it can be constructed as follows: let
$K=\SU(n)$. Then the alcove $\cA$ has $n$ vertices, namely $x_0=0$ and
the fundamental weights $x_i=\omega_i$, $i=1,\ldots,n-1$. Let $\cP$ be
the edge joining $x_0$ and $x_1$. Let $\Lambda=\ZZ\omega_1$. Then
$(\cP,\Lambda)$ with $\Lambda=\ZZ\omega_1$ is a spherical pair.
Indeed,the smooth affine spherical $\SL(n,\CC)$-variety $X=\CC^n$ has
weight monoid $\ZZ_{\ge0}\omega_1$.  This shows that it is a local
model at the vertex $x_0$. The situation in $x_1$ is similar: the
centralizer is still $K=\SU(n)$ but the simple root system is
different, namely
$\alpha_2,\alpha_3,\ldots,\alpha_{n-1},\alpha_n=\alpha_0$. The last
fundamental weight with respect to this system is
$-\omega_1$. Therefore the monoid
$C_{x_1}\cP\cap\Lambda=\ZZ_{\ge0}(-\omega_1)$ has a model, as well,
namely again $\CC^n$. Glued together this yields the spinning
$2n$-sphere.

Eshmatov, \cite{Eshmatov}, has found an analogue of the spinning
$2n$-sphere for the symplectic group. More precisely, he showed that
the quaternionic projective space $\P^n(\HH)$ carries a structure of
a \mf\ $\Sp(2n)$-manifold. Using our theory, this example can be
obtained as follows. Let $K=\Sp(2n)$ and let
$\epsilon_1,\ldots,\epsilon_n$ be the standard basis of the Cartan
subalgebra $\ft$. Let $\cP$ be the line segment joining the origin
$x_0=0$ with $x_1=\half\epsilon_1$. This is an edge of the fundamental
alcove $\cA$. Put $\Lambda:=\ZZ\epsilon_1$. Then the smooth affine
spherical $\Sp(2n,\CC)$-variety $\CC^{2n}$ is a local model in
$x_0$. The other endpoint $x_1$ behaves differently, though. In this
case the simple roots of the centralizer $K_{x_1}$ are
$\alpha_0,\alpha_2,\alpha_3,\ldots,\alpha_n$ which yields
$K_{x_1}=\Sp(2)\times \Sp(2n-2)$. Moreover $-\omega_1$ is now the
fundamental weight of the first factor of $K_{x_1}$. The local model with
weight monoid $\ZZ_{\ge0}(-\omega_1)$ is $\CC^2$ with the second
factor of $K_{x_1}$ acting trivially. This shows that $M$ is obtained by
gluing the open pieces $U_1=\CC^n$ and
\[
  U_2=\Sp(2n)\Times^{\Sp(2)\times \Sp(2n-2)}\CC^2.
\]

This example has been further generalized by Knop-Paulus in
\cite{KnopPaulus}. We keep $K=\Sp(2n)$. Then the vertices of $\cA$ are
$x_k:=\half\sum_{i=1}^k\epsilon_k$ for $k=0,\ldots,n$. Fix $k$ with
$k>0$ and let $\cP_k$ be the line segment joining $x_{k-1}$ and
$x_k$. Let moreover $\Lambda_k:=\ZZ\epsilon_k$. Then one shows as
above that $(\cP_k,\Lambda_k)$ is spherical and it is even possible to
identify the corresponding manifold:

\begin{theorem}

  Let $n,k$ be integers with $1\le k\le n$. Then there is a \mf\
  $\Sp(2n)$-manifold structure on the quaternionic Grassmannian
  $M=\Gr_k(\HH^{n+1})$ with $(\cP_M,\Lambda_M)=(\cP_k,\Lambda_k)$.

\end{theorem}

\begin{proof}

  The open pieces at $x_{k-1}$ and $x_k$, respectively, are the spaces
  \[
    X_1:=\Sp(2n)\times^{H_{k-1}}\CC^{2n-2k+2}\text{ and }
    X_2:=\Sp(2n)\times^{H_k}\CC^{2k}
  \]
  where $H_k:=\Sp(2k)\times \Sp(2n-2k)\subseteq \Sp(2n)$ and they glue
  to a \mf\ manifold $M$. Now recall that $\Sp(2n)$ can also be
  interpreted as the unitary group of $\HH^n$. Then $H_k$ is the
  isotropy group of $\HH^k\subseteq\HH^n$. Therefore $X_2$ can be
  identified with the universal bundle $\widetilde{\Gr}_k(\HH^n)$ over
  the quaternionic Grassmannian $\Gr_k(\HH^n)$. Similarly, $X_1$ is
  isomorphic to $\widetilde{\Gr}_{n-k+1}(\HH^n)$. Now consider the
  space $\HH^{n+1}=\HH^n\oplus\HH$ where $K$ acts on the first
  factor. Let $e:=(0,1)$ be the fixed point. Each element of
  $\widetilde{\Gr}_k(\HH^n)$ can be interpreted as a pair $(L,v)$ with
  $L\in\Gr_k(\HH^n)$ and $v\in L$. Let
  $\Gamma_{L,v}\subseteq\HH^n\oplus\HH$ be the graph of the map
  $L\to\HH:u\mapsto\<u,v\>$. Then the map $(L,v)\mapsto\Gamma_{L,v}$
  identifies $X_2=\widetilde{\Gr}_k(\HH^n)$ with the open subset of
  all $\tilde L\in\Gr_k(\HH^{n+1})$ with $e\not\in \tilde
  L$. Similarly, $X_1$ be identified with the set of all
  $\tilde L\in\Gr_k(\HH^{n+1})$ with $e\not\in\tilde L^\perp$. So
  $\Gr_k(\HH^{n+1})$ is also obtained by gluing $X_1$ and $X_2$. One
  can check using, e.g., \cite{AA1} or \cite{AA2}*{Thm.\ 7.1} that all
  such gluings give diffeomorphic results. So $M\cong\Gr_k(\HH^{n+1})$.
\end{proof}

\subsection*{Surjective momentum maps}

It is interesting to look at \mf\ manifolds $M$ which are in a sense
as big as possible. For us this means that $\cP_M$ is the entire
alcove $\cA$ and $\Lambda_M$ is the weight lattice $P$ of $\VPhi$. In
geometric terms, these are the \mf\ manifolds where the momentum map
is surjective and where the principal isotropy group is trivial.

\begin{proposition}\label{prop:surjective}

  Let $(K,\tau)$ be one of the following three cases:
  \[
    (\SU(n),\id),\quad(\Sp(2n),\id),\quad (\SU(2n+1),k\mapsto\overline
    k)
  \]
  (the last $\tau$ is complex conjugation). Then $(\cA,P)$ is
  spherical, i.e., there is a unique \mf\ $K\tau$-manifold $M$ whose
  momentum map is surjective and such that $K$ acts freely on $M$.

\end{proposition}

\begin{proof} It suffices to find a local model in each of the
  vertices $a$ of $\cA$. For that, each case will be treated
  separately.

  $(K,\tau)=(\SU(n),\id)$: We start with $a=0\in\cP=\cA$. Then $K_a=K$
  and $C_a\cA$ is the dominant Weyl chamber. Therefore, we have to
  show that there is a smooth affine $\SL(n,\CC)$-variety $X_n$ such
  that $\CC[X]=\bigoplus_\chi L(\chi)$ where $\chi$ runs through all
  dominant weights. Such a variety does in general not exist for an
  arbitrary reductive group but it does for $\SL(n,\CC)$, namely
  \[
    X_n:=\begin{cases}
      \SL(n,\CC)\Times^{\Sp(n,\CC)}\CC^n&\text{if $n$ is even},\\
      \SL(n,\CC)/\Sp(n-1,\CC)&\text{if $n$ is odd}.
    \end{cases}
  \]
  Thus $(\cA,\Lambda)$ is spherical in $a=0$. But then it is also
  spherical in all other vertices of $\cA$ since they differ only in a
  translation by an element of the center.

  $(K,\tau)=(\Sp(2n),\id)$: A local model in $a=0$ is
  \[
    Y_n:=\begin{cases}
      \Sp(2n,\CC)\Times^{\Sp(n,\CC)\times \Sp(n,\CC)}\CC^n&\text{if $n$ is even},\\
      \Sp(2n,\CC)\Times^{\Sp(n-1,\CC)\times
        \Sp(n+1,\CC)}\CC^{n+1}&\text{if $n$ is odd}.
    \end{cases}
  \]
  In general, $\cA$ has $n+1$ vertices $x_0=0,x_1,\ldots,x_n$ which
  are enumerated in such a way that $\alpha_i(x_i)\ne0$ where
  $\alpha_0,\ldots,\alpha_n$ are the simple roots. Then the
  centralizer of $x_i$ in $K$ is $L=\Sp(2i)\times \Sp(2n-2i)$. Since
  $\Lambda=\ZZ^n=\ZZ^i\oplus\ZZ^{n-i}$ splits accordingly, the
  manifolds
  \[
    Y_{i,n-i}:=\Sp(2n)\Times^{\Sp(2i)\times \Sp(2n-2i)}(Y_i\times
    Y_{n-i})
  \]
  are the open pieces in $x_i$.

  $(K,\tau)=(\SU(2n+1),\tau)$ with $\tau$ an outer automorphism: Here,
  the Dynkin diagram of $(K,\tau)$ is of type $\sA_{2n}^{(2)}$. In
  this case, $\cA$ has $n+1$ vertices $x_0,\ldots,x_n$ such that the
  centralizer of $x_i$ is $L=\Sp(2i)\times \SO(2n+1-2i)$. It is
  well-known that the coordinate ring of
  \[
    Z_n:=\SO(2n+1,\CC)/GL(n,\CC)
  \]
  contains all irreducible $\SO(2n+1,\CC)$-modules exactly once. So
  \[
    Z_{i,n-i}:=\SU(2n+1)\Times^{\Sp(2i)\times \SO(2n+1-2i)}(Y_i\times
    Z_{n-i})
  \]
  is an open piece in $x_i$.
\end{proof}

\begin{remark}
  Paulus, \cite{Paulus}, has determined all \mf\ manifolds with
  surjective momentum map. Thereby he showed that the manifolds above
  are the only ones where $K$ is simple and the generic isotropy is
  trivial.
\end{remark}

It is also interesting to determine the (global) root system $\Phi_M$
generated by the local root system from \cref{thm:gerbe} using
\cref{prop:TrivialRootCrit}. For that it suffices to calculate its
simple roots, the so called \emph{spherical roots} of $M$. To do this
we use that the spherical roots of the local models are known.

We only treat the case $(K,\tau)=(\SU(n),\|id|)$ in
\cref{prop:surjective}. The simple affine roots of $K$ are
\[
  \alpha_0=1+x_n-x_1,\alpha_1=x_1-x_2,\ldots,\alpha_{n-1}=x_{n-1}-x_n.
\]
The spherical roots of $X_n$ are
$\alpha_1+\alpha_2,\alpha_2+\alpha_3,\ldots,\alpha_{n-2}+\alpha_{n-1}$. For
$n$ odd, see \cite{BP} while the even case is handled in
\cite{LunaModele}. It follows from the comparison results of
\cite{KnopAutoHam}, in particular Thms.\ 3.3 and 9.1, that the
spherical roots of $X_n$ are the simple roots of the local root
systems of $M$. Therefore, the simple roots of $\Phi_M$ are
\[
  1+x_n-x_2,x_1-x_3,x_2-x_4,\ldots,x_{n-2}-x_n,1+x_{n-1}-x_1.
\]
Hence
\[
  \Phi_M\cong
  \begin{cases}
    \sA_{\frac n2-1}^{(1)}\times\sA_{\frac n2-1}^{(1)}&\text{$n$ even,}\\
    \vrule height 14pt width 0pt\sA_{n-1}^{(1)}&\text{$n$ odd.}\\
  \end{cases}
\]
Observe that in the odd case the root systems of $K$ and $M$ are
isomorphic but they are not the same. For example, for $n=3$, i.e.,
$K=\SU(3)$, one gets the picture
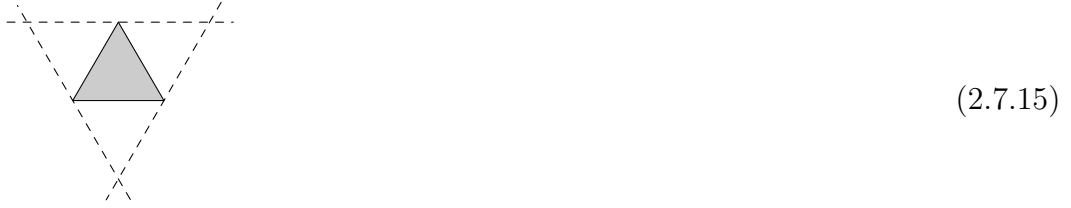
\[\label{eq:fig1}
  \vcenter{\hbox{\begin{tikzpicture}[scale=0.2,line cap=round,line
    join=round,>=triangle 45,x=1.0cm,y=1.0cm]
    \fill[fill=black,fill opacity=0.2] (0.,0.) -- (6.,0.) --
    (3.,5.196152422706632) -- cycle; \draw (0.,0.)-- (6.,0.); \draw
    (6.,0.)-- (3.,5.196152422706632); \draw (3.,5.196152422706632)--
    (0.,0.); \draw [dash pattern=on 3pt off 3pt] (-4.32,5.2)--
    (10.56,5.196152422706632); \draw [dash pattern=on 3pt off 3pt]
    (9.762262054223738,6.516429029303968)--
    (2.19943667558704,-6.582768775266124); \draw [dash pattern=on 3pt
    off 3pt] (3.785563324412959,-6.556788013152589)--
    (-3.6393197678831375,6.303486742963368);
  \end{tikzpicture}}}
\]
where the gray triangle denotes $\cP=\cA$ and the axes of the simple
reflections of $\Phi_M$ are marked by dashed lines. There is also
something to be observed in the even case: here all roots of $\Phi_M$
are perpendicular to the vector
$\delta=(1,-1,\ldots,1,-1)\in\Vfa$. Let $f$ be an affine linear
function of $\fa$ with $\nabla f=\delta$. Then $f$ is $W_M$-invariant
and $t\mapsto\epsilon(tf)$ from \eqref{eq:epsilon} defines a
1-parameter subgroup of automorphisms of $M$. Since $\delta$ lies in
the weight lattice, this action factors through an action of an
$1$-dimensional torus. Thus, the $\SU(n)$-action on $M$ extends to an
$U(1)\times \SU(n)$-action.

\subsection*{Inscribed triangles}

For the last example, we toyed with triangles inscribed in a
triangular alcove. Here are some examples of spherical pairs
$(\cP,\Lambda)$:
\[
  \begin{array}{|l|c|c|c|c|c|}
    \hline
    K&\SU(3)&\SU(3)&\Sp(4)&\Sp(4)&\sG_2\\
    \hline
    \begin{matrix}\cP\subseteq\cA\\\ \\\ \end{matrix}&\begin{tikzpicture}[scale=0.3,line cap=round,line join=round,>=triangle 45,x=1.0cm,y=1.0cm]
      \fill[fill opacity=0.2] (3.,0.) -- (4.5,2.598076211353316) --
      (1.5,2.598076211353316) -- cycle; \draw (0.,0.)-- (6.,0.); \draw
      (6.,0.)-- (3.,5.196152422706632); \draw (3.,5.196152422706632)--
      (0.,0.); \draw (3.,0.)-- (4.5,2.598076211353316); \draw
      (4.5,2.598076211353316)-- (1.5,2.598076211353316); \draw
      (1.5,2.598076211353316)-- (3.,0.); \draw
      (1.5,2.598076211353316)-- (3.,0.); \draw (3.,0.)--
      (4.5,2.598076211353316); \draw (4.5,2.598076211353316)--
      (1.5,2.598076211353316);
    \end{tikzpicture}
            &\begin{tikzpicture}[scale=0.3,line cap=round,line join=round,>=triangle 45,x=1.0cm,y=1.0cm]
              \fill[fill opacity=0.2] (4.,0.) --
              (4.,3.464101615137755) -- (1.,1.7320508075688774) --
              cycle; \draw (0.,0.)-- (6.,0.); \draw (6.,0.)--
              (3.,5.196152422706632); \draw (3.,5.196152422706632)--
              (0.,0.); \draw (4.,0.)-- (4.,3.464101615137755); \draw
              (4.,3.464101615137755)-- (1.,1.7320508075688774); \draw
              (1.,1.7320508075688774)-- (4.,0.); \draw
              (1.,1.7320508075688774)-- (4.,0.); \draw (4.,0.)--
              (4.,3.464101615137755); \draw (4.,3.464101615137755)--
              (1.,1.7320508075688774);
            \end{tikzpicture}
                   &\begin{tikzpicture}[scale=0.3,line cap=round,line join=round,>=triangle 45,x=1.0cm,y=1.0cm]
                     \fill[fill opacity=0.2] (3.,0.) -- (6.,3.) --
                     (3.,3.) -- cycle; \draw (0.,0.)-- (6.,0.); \draw
                     (6.,0.)-- (6.,6.); \draw (6.,6.)-- (0.,0.); \draw
                     (3.,0.)-- (6.,3.); \draw (6.,3.)-- (3.,3.); \draw
                     (3.,3.)-- (3.,0.); \draw (3.,0.)-- (6.,3.); \draw
                     (6.,3.)-- (3.,3.); \draw (3.,3.)-- (3.,0.);
                   \end{tikzpicture}
                          &\begin{tikzpicture}[scale=0.3,line cap=round,line join=round,>=triangle 45,x=1.0cm,y=1.0cm]
                            \fill[fill opacity=0.2] (4.,0.) -- (6.,2.)
                            -- (2.,2.) -- cycle; \draw (0.,0.)--
                            (6.,0.); \draw (6.,0.)-- (6.,6.); \draw
                            (6.,6.)-- (0.,0.); \draw (4.,0.)-- (6.,2.);
                            \draw (6.,2.)-- (2.,2.); \draw (2.,2.)--
                            (4.,0.); \draw (2.,2.)-- (4.,0.); \draw
                            (4.,0.)-- (6.,2.); \draw (6.,2.)-- (2.,2.);
                          \end{tikzpicture}
                                 &\begin{tikzpicture}[scale=0.3,line cap=round,line join=round,>=triangle 45,x=1.0cm,y=1.0cm]
                                   \fill[fill opacity=0.2] (4.,0.) --
                                   (6.,1.1547005383792512) --
                                   (4.,2.3094010767585025) -- cycle;
                                   \draw (0.,0.)-- (6.,0.); \draw
                                   (6.,0.)-- (6.,3.464101615137754);
                                   \draw (6.,3.464101615137754)--
                                   (0.,0.); \draw (4.,0.)--
                                   (6.,1.1547005383792512); \draw
                                   (6.,1.1547005383792512)--
                                   (4.,2.3094010767585025); \draw
                                   (4.,2.3094010767585025)-- (4.,0.);
                                   \draw (4.,0.)--
                                   (6.,1.1547005383792512); \draw
                                   (6.,1.1547005383792512)--
                                   (4.,2.3094010767585025); \draw
                                   (4.,2.3094010767585025)-- (4.,0.);
                                 \end{tikzpicture}\\
    \Lambda&P\text{ or }R&R&R&R&R\\
    \hline
  \end{array}
\]
We make no claim of completeness. In particular, we considered only
untwisted groups. The letters $P$ and $R$ denote the weight and the
root lattice of $K$, respectively. At each vertex, the complexified
centralizer $L$ is isogenous to $\SL(2,\CC)\times\CC^*$. Then one can
show that the local models are either of the form $X=\SL(2,\CC)/\mu_n$
in case $\cP$ touches $\cA$ in form of a reflection and
$X=\SL(2,\CC)\times^{\CC^*}\CC$ otherwise.

\begin{remark}

  As communicated to me by Eckhart Meinrenken, the first triangle has
  also been found by Chris Woodward (unpublished).

\end{remark}

\end{document}